\documentclass[12pt]{amsart}
\usepackage{graphicx,amssymb}
\usepackage[usenames]{color}
\usepackage{amsthm,amsfonts,amsmath,amssymb,latexsym,epsfig,mathrsfs,yfonts,marvosym,latexsym,epsfig}
\AtEndDocument{\vfill\eject\batchmode}
\DeclareMathAlphabet\oldmathcal{OMS}        {cmsy}{b}{n}
\SetMathAlphabet    \oldmathcal{normal}{OMS}{cmsy}{m}{n}
\DeclareMathAlphabet\oldmathbcal{OMS}       {cmsy}{b}{n}
\usepackage[active]{srcltx}
\usepackage{eucal}

\newtheorem{theorem}{Theorem}[section]
\newtheorem{lemma}[theorem]{Lemma}
\newtheorem{proposition}[theorem]{Proposition}
\newtheorem{corollary}[theorem]{Corollary}

\newtheorem{def/prop}[theorem]{Definition/Proposition}

\theoremstyle{definition}
\newtheorem{definition}[theorem]{Definition}
\newtheorem{remark}[theorem]{Remark}
\newtheorem*{ack}{Acknowledgements}
\newtheorem{example}{Example}[section]

\DeclareSymbolFont{bbold}{U}{bbold}{m}{n}
\DeclareSymbolFontAlphabet{\mathbbold}{bbold}

\def\BOne{\mathchoice{\scalebox{1.16}{$\displaystyle\mathbbold 1$}}{\scalebox{1.16}{$\textstyle\mathbbold 1$}}{\scalebox{1.16}{$\scriptstyle\mathbbold 1$}}{\scalebox{1.16}{$\scriptscriptstyle\mathbbold 1$}}}
\def\fract#1#2{\raise4pt\hbox{$ #1 \atop #2 $}}

\def\bbc{{\mathbb C}}

\def\bbn{{\mathbb N}}

\def\bbp{{\mathbb P}}
\def\bbq{{\mathbb Q}}
\def\bbr{{\mathbb R}}

\def\bbt{{\mathbb T}}

\def\bbz{{\mathbb Z}}

\def\gra{\alpha}
\def\grb{\beta}

\def\grg{\gamma}
\def\gri{\iota}
\def\grk{\kappa}

\def\gro{\omega}

\def\grz{\zeta}
\def\grD{\Delta}

\def\grG{\Gamma}

\def\grL{\Lambda}
\def\grO{\Omega}

\def\bfl{{\bf l}}
\def\bfm{{\bf m}}

\def\bfv{{\bf v}}
\def\bfw{{\bf w}}

\def\bfV{{\bf V}}

\def\bfS{{\bf S}}

\def\cala{{\mathcal A}}

\def\calc{{\mathcal C}}
\def\cald{{\mathcal D}}
\def\cale{{\mathcal E}}
\def\calf{{\mathcal F}}
\def\calg{{\mathcal G}}
\def\calh{{\mathcal H}}
\def\cali{{\mathcal I}}

\def\calm{{\mathcal M}}

\def\calo{{\mathcal O}}
\def\calp{{\mathcal P}}
\def\calq{{\mathcal Q}}
\def\calr{{\mathcal R}}
\def\cals{{\oldmathcal S}}

\def\calw{{\mathcal W}}

\def\calz{{\mathcal Z}}
\def\calS{{\mathcal S}}

\def\gh{{\mathfrak h}}

\def\gr{{\mathfrak r}}

\def\gt{{\mathfrak t}}
\def\gu{{\mathfrak u}}

\def\gz{{\mathfrak z}}

\def\gA{{\mathfrak A}}

\def\gC{{\mathfrak C}}

\def\gR{{\mathfrak R}}

\def\tTheta{\tilde{\Theta}}
\def\lcm{{\rm lcm}}

\def\ttheta{\tilde{\theta}}
\def\gtz{\tilde{\mathfrak z}}

\def\<{\langle}
\def\>{\rangle}
\def\ra#1{\to}

\def\g{\gamma}

\def\fract#1#2{\raise4pt\hbox{$ #1 \atop #2 $}}
\def\decdnar#1{\phantom{\hbox{$\scriptstyle{#1}$}}
\left\downarrow\vbox{\vskip15pt\hbox{$\scriptstyle{#1}$}}\right.}

\def\lcm{{\rm lcm}}

\def\hook{\mathbin{\hbox to 6pt{%
                 \vrule height0.4pt width5pt depth0pt
                 \kern-.4pt
                 \vrule height6pt width0.4pt depth0pt\hss}}}

\begin{document}

\title{The $S^3_\bfw$ Sasaki Join Construction}

\author[Charles Boyer]{Charles P. Boyer}
\address{Charles P. Boyer, Department of Mathematics and Statistics,
University of New Mexico, Albuquerque, NM 87131.}
\email{cboyer@math.unm.edu} 
\author[Christina T{\o}nnesen-Friedman]{Christina W. T{\o}nnesen-Friedman}
\address{Christina W. T{\o}nnesen-Friedman, Department of Mathematics, Union
College, Schenectady, New York 12308, USA }
\email{tonnesec@union.edu}
\thanks{The authors were partially supported by grants from the
Simons Foundation, CPB by (\#519432), and CWT-F by (\#422410)}

\begin{abstract}
The main purpose of this work is to generalize the $S^3_\bfw$ Sasaki join construction $M\star_\bfl S^3_\bfw$ described in the authors’ 2016 paper when the Sasakian structure on $M$ is regular, to the general case where the Sasakian structure is only quasi-regular. This gives one of the main results, Theorem \ref{admjoincsc}, which describes an inductive procedure for constructing Sasakian metrics of constant scalar curvature. In the Gorenstein case ($c_1(\cald)=0$) we construct a polynomial whose coeffients are linear in the components of $\bfw$ and whose unique root in the interval $(1,\infty)$ completely determines the Sasaki-Einstein metric. In the more general case  we apply our results to prove that there exists infinitely many smooth 7-manifolds each of which admit infinitely many inequivalent contact structures of Sasaki type admitting constant scalar curvature Sasaki metrics (see Corollary \ref{infcon}). We also discuss the relationship with a recent paper of Apostolov and Calderbank as well as the relation with K-stability.

\end{abstract}

\maketitle
\vspace{-7mm}



\section*{Introduction}\label{intro}
It is often said that Sasaki geometry is the odd dimensional version of K\"ahler geometry; nevertheless, there are substantial differences. The most obvious difference concerns products. The product of K\"ahler manifolds (orbifolds) is K\"ahler. This simple formulation does not occur in Sasaki geometry for dimensional reasons. The closest analogue is the so-called Sasaki join operation which amounts to a product of the corresponding transverse K\"ahler structures and which by now has been developed in some detail \cite{BG00a,BGO06,BG05,HeSu12b,BoTo14a,BHLT16}. However, the Sasaki join is much richer as we shall now discuss. The K\"ahler and Sasaki cones are analogous in the sense that they provide a family of K\"ahler (Sasaki) structures that are associated to a fixed underlying complex, (respectively, CR) structure. They are, however, very different in nature. The K\"ahler cone is a cone in the vector space of 2-dimensional de Rham cohomology classes; whereas, the Sasaki cone consists of nowhere vanishing Killing vector fields. Now consider the following examples. First, let $(N,\gro_N)$ be a K\"ahler-Einstein manifold or orbifold and consider the product with the standard complex projective line $(\bbc\bbp^1,\gro_{FS})$ with its Fubini-Study metric. Then there is a choice of positive integers $l_1,l_2$ such that $(N\times \bbc\bbp^1,l_1\gro_N+l_2\gro_{FS})$ is K\"ahler-Einstein. A similar analogue holds for the Sasaki join $M\star_\bfl S^3$. Now if we replace $(\bbc\bbp^1,\gro_{FS})$ by a weighted projective line $(\bbc\bbp^1[\bfw],\gro_{ext})$ with an extremal K\"ahler orbifold metric, there is no pair of integers $l_1,l_2$ nor deformation in the K\"ahler cone that produces a K\"ahler-Einstein orbifold metric on the product $N\times \bbc\bbp^1[\bfw]$. However, as shown explicitly in \cite{BoTo14a} when $M$ is a regular Sasaki-Einstein manifold, for each weight vector $\bfw=(w_1,w_2)$ there are positive integers $l_1,l_2$ such that the corresponding Sasaki join $M\star_\bfl S^3_\bfw$ admits a Sasaki-Einstein metric in the Sasaki cone. The purpose of this paper is to generalize the study of the $S^3_\bfw$ join construction in \cite{BoTo14a} to the case where $M$ in $M_{\bfl,\bfw}=M\star_\bfl S^3_\bfw$ is an arbitrary quasi-regular Sasaki manifold or orbifold. This has important consequences for GIT stability in the Sasaki category that does not hold in the K\"ahler category. Namely, in a certain sense the $S^3_\bfw$ Sasaki join operation preserves K-stability for all weight vectors $\bfw$. This amounts to the join of an unstable, but relatively K-stable manifold, producing a K-stable Sasaki manifold.

The complexity of the multiplicative structure described above, the $l$-monoid, adds to the complexity of the moduli space of Sasakian structures which increases drastically with dimension. In dimension 3 things are well understood and there is a classifcation of Sasakian structures \cite{Bel01,BG05}. In dimension 5 it is already much more complex. See Chapter 10 of \cite{BG05} and Section 6 of \cite{Boy18} for a recent update. There is, however, a classification of simply connected compact 5-manifolds due to Smale and Barden, not all of which admit Sasakian structures. Precisely which ones do and which don't is still an open question. Reducibility essentially begins in dimension 5 where the so-called cone decomposable Sasakian structures are lens space bundles over Riemann surfaces \cite{BoTo13,BHLT16}. See Corollary \ref{5cor} below. In dimension 7 things become yet even more complex, since there are many irreducible Sasaki 5-manifolds. Given this situation it seems prudent to understand which Sasakian structures can be built up from lower dimensional ones and how stability propagates. This is precisely what the join accomplishes. 

\begin{ack}
We thank Vestislav Apostolov and David Calderbank for many very helpful discussions and their interest in our work. We also thank the anonymous referees for their careful reading of our manuscript, and especially for catching an error in an earlier version of Proposition \ref{reglem} and for clarifications in Section 5. 
\end{ack}

\section{Brief Review of Sasaki Geometry}
Recall that a Sasakian structure on a contact manifold $M^{2n+1}$ of dimension $2n+1$ is a special type of contact metric structure $\cals=(\xi,\eta,\Phi,g)$ with underlying almost CR structure $(\cald,J)$ where $\eta$ is a contact form such that $\cald=\ker\eta$, $\xi$ is its Reeb vector field, $J=\Phi |_\cald$, and $g=d\eta\circ (\BOne \times\Phi) +\eta\otimes\eta$ is a Riemannian metric. $\cals$ is a Sasakian structure if $\xi$ is a Killing vector field and the almost CR structure is integrable, i.e. $(\cald,J)$ is a CR structure. We refer to \cite{BG05} for the fundamentals of Sasaki geometry. We call $(\cald,J)$ a {\it CR structure of Sasaki type}, and $\cald$ a {\it contact structure of Sasaki type}. We shall always assume that the Sasaki manifold $M^{2n+1}$ is compact and connected.

\subsection{The Sasaki Cone}
Within a fixed contact CR structure $(\cald,J)$ there is a conical family of Sasakian structures known as the Sasaki cone. We are also interested in a variation within this family. To describe the Sasaki cone we fix a Sasakian structure $\cals_o=(\xi_o,\eta_o,\Phi_o,g_o)$ on $M$ whose underlying CR structure is $(\cald,J)$ and let $\gt$ denote the Lie algebra of a maximal torus $\bbt$ in the automorphism group of $\cals_o$. The {\it (unreduced) Sasaki cone} \cite{BGS06} is defined by
\begin{equation}\label{sascone}
\gt^+(\cald,J)=\{\xi\in\gt~|~\eta_o(\xi)>0~\text{everywhere on $M$}\},
\end{equation}
which is a cone of dimension $k\geq 1$ in $\gt$. The reduced Sasaki cone $\grk(\cald,J)$ is $\gt^+(\cald,J)/\calw$ where $\calw$ is the Weyl group of the maximal compact subgroup of $\gC\gR(\cald,J)$ which, as mentioned previously, is the moduli space of Sasakian structures with underlying CR structure $(\cald,J)$. However, it is more convenient to work with the unreduced Sasaki cone $\gt^+(\cald,J)$. It is also clear from the definition that $\gt^+(\cald,J)$ is a cone under the transverse scaling defined by
\begin{equation}\label{transscale}
\cals=(\xi,\eta,\Phi,g)\mapsto \cals_a=(a^{-1}\xi,a\eta,g_a),\quad g_a=ag+(a^2-a)\eta\otimes\eta, \quad a\in\bbr^+.
\end{equation}
So $\gt^+(\cald,J)$ is a cone and since the Reeb vector field $\xi$ is Killing $\dim\gt^+(\cald,J)\geq 1$. Moreover, it follows from contact geometry that $\dim\gt^+(\cald,J)\leq n+1$. When $\dim\gt^+(\cald,J)=n+1$ we have a toric contact manifold of Reeb type studied in \cite{BM93,BG00b,Ler02a,Ler04,Leg10,Leg16}. In this case there is a strong connection between the geometry and topology of $(M,\cals)$ and the combinatorics of $\gt^+(\cald,J)$. Much can also be said in the complexity 1 case ($\dim\gt^+(\cald,J)=n$) \cite{AlHa06}.
 
We often have need to deform the contact structure $\cald\mapsto \cald_\varphi$ by a contact isotopy $\eta\mapsto \eta +d^c\varphi$ where $\varphi\in C^\infty(M)^\bbt$ is a smooth function invariant under the torus $\bbt$. We note that the Sasaki cone $\gt^+(\cald,J)$ is invariant under such contact isotopies, that is, $\gt^+(\cald_\varphi,J_\varphi)= \gt^+(\cald,J)$. For each $\varphi\in C^\infty(M)^\bbt$, $\cald_\varphi\longrightarrow TM$ gives a splitting of the exact sequence
$$0\longrightarrow L_{\xi_o}\longrightarrow TM\longrightarrow Q\longrightarrow 0$$
with $J_\varphi=\Phi |_{{\cald}_{\varphi}}$.  Furthermore, each choice of Reeb vector field $\xi\in\gt^+(\cald,J)$ gives rise to an infinite dimensional contractible space $\calS(M,\xi)$ of Sasakian structures \cite{BG05}. We shall often make such a choice $\cals=(\xi,\eta,\Phi,g)\in\calS(M,\xi)$ and identify it with the element $\xi\in\gt^+(\cald,J)$.

\section{The $S^3$-Join Construction}
The join construction is the Sasaki analogue for products in K\"ahler geometry. It was first described in the context of Sasaki-Einstein manifolds \cite{BG00a}, but then developed more generally in \cite{BGO06}, see also Section 7.6.2. of \cite{BG05}. Its relation to the de Rham decomposition Theorem and reducibility questions \cite{HeSu12b} were studied in \cite{BHLT16}. We consider the set $\cals\calo$ (respectively, $\cals\calm$) of all Sasaki orbifolds (manifolds), respectively. $\cals\calo$ is the object set of a groupoid whose morphisms are orbifold diffeomorphisms. Moreover, $\cals\calo$ is graded by dimension and has an additional multiplicative structure described in Section 2.4 of \cite{BHLT16} which we call the {\it Sasaki $l$-monoid}. The positive irreducible generators in dimension 3 are of the form $\star_\bfl S^3_\bfw$ where the weight vector $\bfw=(w^0,w^\infty)$ has relatively prime components. We apply this to the submodule $\cals\calm$ which is not an ideal in general. However, as we shall see $\star_\bfl S^3_\bfw$ does give a map $\cals\calm\ra{1.8}\cals\calm$ when certain conditions hold. We have used the join construction \cite{BoTo13,BoTo14P,BoTo14a} to produce new cscS metrics from known cscS metrics. In most cases (see  \cite{BoTo18c} for an exception) in constructing the join $M\star_\bfl S^3_\bfw$ we have assumed that the Sasakian structure on $M$ is regular. This is a restrictive assumption which is not necessary. We only need the Sasakian structure $\cals$ to be quasiregular in which case it is described by an $S^1$ orbibundle $M\longrightarrow N$ which has order $\Upsilon_\cals$. The procedure as outlined for example in \cite{BoTo14a} essentially works in the more general quasiregular case; however, one needs to proceed with a bit  more care at several stages, in particular, in handling the smoothness conditions. Before discussing these conditions we briefly discuss the orbibundle $M\longrightarrow N$. We assume that $N$ is a normal (compact) projective algebraic variety with a fixed orbifold structure which we write as the pair $(N,\grD_N)$ where $\grD_N$ is a sum of irreducible branch divisors, viz. 
$$\grD_N= \sum_{j=1}^k\bigl(1-\frac{1}{m_j}\bigr)D_j$$
where $k=\dim {\rm Div}(N)$,  the group of Weil divisors on $N$, $m_j$ are the ramification indices, and $D_j\in {\rm Div}(N)$. The algebraic singular locus is also given an orbifold structure. They are finite cyclic quotient singularities.

By Theorem 4.3.15 of \cite{BG05} the isomorphism classes of $S^1$ orbibundles over $(N,\grD_N)$ are uniquely determined by their orbifold first Chern class denoted by $c_1^{orb}(M/N)\in H^2(N,\bbq)$. Furthermore, when $M$ has a Sasakian structure, $c_1^{orb}(M/N)$ is an orbifold K\"ahler class with K\"ahler form denoted by $\gro_N$.

\begin{definition}\label{primdef}
We say that a compact Sasaki manifold $M=(M,\cals)$ with $\cals$ quasi-regular is {\it Sasaki primitive} or just {\it primitive} if there is no nontrivial finite cyclic Sasaki cover $\tilde{M}$ of $M$. By a finite cyclic Sasaki cover we mean a Sasaki manifold $(\tilde{M},\tilde{\cals})$ such that $M=\tilde{M}/G$ with $G$ a finite cyclic subgroup of the circle group induced by the Reeb vector field $\xi$. Here $\cals$ is the Sasakian structure on the quotient $M$ that is naturally induced from $\tilde{\cals}$.
\end{definition}

\begin{remark}\label{primrem}
It is well known that a compact regular Sasaki manifold $(M,\cals)$ is primitive if and only if the induced K\"ahler class $[\gro_N]$ is a primitive element of $H^2(N,\bbz)$. Moreover, there is a family $\{(M_n,\cals_n)\}_{n\in\bbz^+}$ of Sasaki manifolds with $(M_1,\cals_1)=(M,\cals)$ associated to the primitive K\"ahler class $[\gro_N]$ whose members have K\"ahler class $n[\gro_N]$ and $M_n=M/\bbz_n$ where $\bbz_n$ acts freely on $M$ and there is a short exact sequence
$$0\longrightarrow \pi_1(M)\longrightarrow \pi_1(M_n)\longrightarrow \bbz_n\longrightarrow 1.$$
However, an analogous result does not hold for quasi-regular Sasaki manifolds unless $\gcd(n,\Upsilon_N)=1$ where $\Upsilon_N$ is the order of the orbifold $N$. This is easily seen by the following example.  Consider $(N,\grD)$ with $N$ smooth and choose a primitive K\"ahler class $[\gro_N]\in H^2_{orb}(N,\bbz)$. Then $\Upsilon_N[\gro_N]$ is a primitive K\"ahler class in $H^2(N,\bbz)$ and as such the total space of the corresponding $S^1$ bundle over $N$ has a regular primitive Sasakian structure, but $\Upsilon_N[\gro_N]$ is clearly not primitive as an orbifold K\"ahler class. 
\end{remark}



\subsection{The Construction}
We now turn to the $S^3_\bfw$ join construction. Let $M$ be a compact quasi-regular Sasaki manifold whose $S^1$ action generated by its Reeb vector field is denoted by $\cala(\theta,x)$. The assumption for the rest of the paper is that $c_1^{orb}(M/N) = [\gro_N]\in H^2(N,\bbq)$ is a primitive orbifold K\"ahler class, where $\gro_N$ denotes the transverse K\"ahler form 
on $N$. Likewise, $c_1^{orb}(S^3_\bfw/\bbc\bbp^1[\bfw]) = [\omega_\bfw]=[\frac{\omega_{FS}}{w^0w^\infty}]$, where $\omega_{FS}$ is the K\"ahler form of the standard Fubini-Study metric on $\bbc\bbp^1$ such that $[\omega_{FS}] \in H^2(\bbc\bbp^1,\bbz)$ is primitive and $\omega_\bfw$ is the transverse extremal K\"ahler form on $\bbc\bbp^1[\bfw]$ of the canonical extremal Sasaki structure on $S^3_\bfw$.

For any pair of relatively prime positive integers $l^0,l^\infty$ we define the $S^3_\bfw$ join with $M$ as the quotient $M_{\bfl,\bfw}$ of
\begin{equation}\label{S3joineqn}
S^1\longrightarrow M\times S^3\longrightarrow M_{\bfl,\bfw}
\end{equation}
where the $S^1$ action is given by 
\begin{equation}\label{S1action}
(x;z_1,z_2)\mapsto (\cala(l^\infty\theta,x);e^{-il^0w^0\theta}z_1,e^{-il^0w^\infty\theta}z_2)
\end{equation}
with $|z_1|^2+|z_2|^2=1$. The join is then constructed from the following  commutative diagram
\begin{equation}\label{s2comdia}
\begin{matrix}  M\times S^3_\bfw &&& \\
                          &\searrow\pi_L && \\
                          \decdnar{\pi_{2}} && M_{\bfl,\bfw} &\\
                          &\swarrow\pi_1 && \\
                           (N,\grD_N)\times\bbc\bbp^1[\bfw] &&& 
\end{matrix} 
\end{equation}
where the $\pi_2$ is the product of the projections of the standard Sasakian projections $\pi_M:M\ra{1.6}  (N,\grD_N)$ and $S^3_\bfw\ra{1.6} \bbc\bbp^1[\bfw]$. The circle action \eqref{S1action} on $M\times S^3$ is generated by the vector field 
\begin{equation}\label{Lvec}
L_{\bfl,\bfw}=\frac{1}{2l^0}\xi_M-\frac{1}{2l^\infty}\xi_\bfw, 
\end{equation}
and its quotient orbifold $M_{\bfl,\bfw}$, which is called the $\bfl=(l^0,l^\infty)$-join of $M$ and $S^3_\bfw$ and denoted by $M_{\bfl,\bfw}=M\star_\bfl
S^3_\bfw$, has a naturally induced quasi-regular Sasakian structure $\cals_{\bfl,\bfw}$ with contact 1-form $\eta_{\bfl,\bfw}$ and Reeb vector field
\begin{equation}\label{Reebjoin}
\xi_{\bfl,\bfw}=\frac{1}{2l^0}\xi_M+\frac{1}{2l^\infty}\xi_\bfw.
\end{equation}
On the join $M_{\bfl,\bfw}=M\star_\bfl S^3_\bfw$ there is a direct sum decomposition of the tangent bundle \cite{HeSu12b, BHLT16}
\begin{equation}\label{tbunsplit}
TM_{\bfl,\bfw}=\cald_1\oplus \cald_2\oplus L_\xi
\end{equation}
which for $i=1,2$ gives rise to two foliations $\cale_i=\cald_i\oplus L_\xi$ of $M_{\bfl,\bfw}$ whose leaves are totally geodesic Sasaki submanifolds isomorphic, up to transverse scaling, as Sasakian structures to $M$ and $S^3_\bfw$, respectively. Note that the intersection $\cale_1\cap \cale_2$ is just $\calf_\xi$, and both the transverse metric $g^T$ and the contact bundle $\cald_{\bfl,\bfw}=\ker\eta_{\bfl,\bfw}$ split as direct sums. The natural numbers $l^0,l^\infty,w^0,w^\infty$ are generally contact invariants. Since the CR-structure $(\cald_{\bfl,\bfw},J_{\bfw})$ is the horizontal lift of  the complex structure on $N\times\bbc\bbp^1[\bfw]$, this splits as well. The choice of $\bfw$ determines the transverse complex structure $J$.

\subsection{Some Elementary Topology}
Generally, the join is an orbifold; however, from Proposition 7.6.6 in \cite{BG05} we have
\begin{lemma}\label{smoothlem}
If $(M,\cals)$ is a quasi-regular Sasaki manifold of order $\Upsilon_\cals$, then the join $M_{\bfl,\bfw}=M\star_\bfl S^3_\bfw$ is smooth if and only if 
\begin{equation}\label{adcond}
\gcd(l^\infty\Upsilon_{\cals},l^0w^0w^\infty)=1
\end{equation}
where the order $\Upsilon_\cals$ is precisely the order $\Upsilon_N$ of the quotient orbifold $N$, i.e. $\Upsilon_\cals=\Upsilon_N$. 
\end{lemma}

From the long exact homotopy sequence of the fibration \eqref{S3joineqn} and the well known Hurewicz Theorem we deduce

\begin{proposition}\label{elemtopprop}
Let $M$ be a simply connected quasi-regular Sasaki manifold with $\pi_2(M)=\bbz^k$. Then 
\begin{enumerate}
\item $M_{\bfl,\bfw}$ is simply connected;
\item $\pi_2(M_{\bfl,\bfw})=\bbz^{k+1}$;
\item $H_2(M_{\bfl,\bfw},\bbz)=\bbz^{k+1}$.
\end{enumerate}
\end{proposition}

The join construction provides a family of contact structures of Sasaki type on a family of smooth manifolds whose cohomology ring can, in principle, be computed if one knows the cohomology ring of $M$.  

\subsection{The $\bfw$-Sasaki Cone $\gt^+_\bfw$}
The Sasakian structures that are most accessible through this construction are the elements of the 2-dimensional subcone $\gt^+_\bfw$ of the Sasaki cone $\gt^+_{\bfl,\bfw}$ known as the {\it $\bfw$-subcone}. If $N$ has no Hamiltonian symmetries then $\gt^+=\gt^+_\bfw$. On $M\times S^3$ we have the Lie algebra $\gt_M\oplus \gt_2$ which induces the Lie algebra isomorphism
\begin{equation}\label{t+add}
\gt_{M,\bfl,\bfw}\approx (\gt_M\oplus \gt_2)/\{L_{\bfl,\bfw}\}
\end{equation}
on $M_{\bfl,\bfw}$ where $\{L_{\bfl,\bfw}\}$ denotes the one dimensional ideal generated by the vector field $L_{\bfl,\bfw}$. 
The Lie algebra $\gt_M$ can be written as $\gt_M=\{\xi_M\} +\gh_M$ where $\gh_M$ is the horizontal lift to $M$ of a maximal commutative Lie algebra of Hamilonian vector fields on $N$. Likewise $\gt_2=\{\xi_\bfw\} +\gh_\bfw$ where $\gh_\bfw$ is the horizontal lift of the 1-dimensional Lie algebra generated by a Killing vector field on $\bbc\bbp^1[\bfw]$. Now the foliations $\cale_i$ on $M_{\bfl,\bfw}$ give rise to Lie algebra monomorphisms 
$$\gri_M:\gt_M\longrightarrow \gt_{M,\bfl,\bfw},\qquad \gri_2:\gt_2\longrightarrow \gt_{M,\bfl,\bfw}$$
such that $\gri_M(\gt_M)\cap\gri_2(\gt_2)=\{\xi_{\bfl,\bfw}\}$. This together with \eqref{t+add} gives

\begin{lemma}\label{torjoinlem}
There is an isomorphism of Abelian Lie algebras
$$\gt_{M,\bfl,\bfw}=\{\xi_{\bfl,\bfw}\} + \gh_M +\gh_\bfw.$$
\end{lemma}

We put $\gt_\bfw=\gri_2(\gt_2)$ and this identifies the Sasaki cone $\gt^+_2$ of $S^3$ with the 2 dimensional subcone $\gt^+_\bfw$ of $\gt^+_{M,\bfl,\bfw}$. Similarly, the Sasaki cone of $M$ is identified with a subcone $\gt^+_M$ of $\gt^+_{M,\bfl,\bfw}$. These two subcones intersect along the ray $\gr_{\bfl,\bfw}$ generated by $\xi_{\bfl,\bfw}$.

 It is not difficult to show that
\begin{lemma}\label{wsubcone}
The $\bfw$-subcone $\gt^+_\bfw$ is invariant under the adjoint action of $\gA\gu\gt_0(M_{\bfl,\bfw})$ on its Lie algebra.
\end{lemma}

\subsection{The Quotient Orbifolds}
We are interested in Sasakian structures $\cals$ in the $\bfw$-subcone $\gt^+_\bfw$, in particular quasi-regular Sasakian structures, that is, Reeb vector fields $\xi_\bfv$ that lie on the integer lattice $\grL_\bfw \subset \gt^+_\bfw$. They are completely determined by the pair $(v^0,v^\infty)$ of relatively prime positive integers. This pair then determines the branch divisors $D^0,D^\infty$ which are the zero sections and infinity sections of an orbibundle $(S_n,\grD_{m,N})\longrightarrow (N,\grD_N)$ with ramification indices $\bfm=(m^0,m^\infty)=(mv^0,mv^\infty)$ through the equations $s=\gcd(l^\infty,w^0v^\infty-w^\infty v^0)$ and $l^\infty=ms$. Before giving the generalization of Theorem 3.8 of \cite{BoTo14a} we describe, following \cite{GhKo05}, what is meant by the orbibundle 
$$\pi^{orb}: (S_n,\grD_{m,N})\longrightarrow (N,\grD_N).$$
First note that for $n \in \bbz$, by $S_n$ we mean the (total space of) the projective bundle \newline $\bbp(\BOne\oplus L_n) \longrightarrow (N,\grD_N)$, where $L_n \longrightarrow (N,\grD_N)$ is an orbi-line bundle such that $c_1^{orb}(L_n)=n [\omega_N]$. Set theoretically $\pi^{orb}$ is a holomorphic map $\pi:S_n\longrightarrow N$ of normal reduced complex spaces, but in the orbifold category $\pi^{orb}$ pullsback the branch divisor $\grD_N$ to $S_n$ giving the orbifold structure on $S_n$ with branch divisor $\grD_{m,N}=\grD_\bfm+\pi^{-1}(\grD_N)$ such that the generic\footnote{By generic fiber we mean the fiber over a point in the orbifold regular locus of $N$.} fibers of $\pi^{orb}$ have the orbifold structure $\bbc\bbp^1[\bfv]/\bbz_m=(\bbc\bbp^1,\grD_\bfm)$. We have the commutative diagram
\begin{equation}\label{galoiscomdia}
\begin{matrix}  (S_n,\grD_{\bfm,N}) &\fract{\pi^{orb}}{\longrightarrow} & (N,\grD_N) \\
                       \decdnar{} && \decdnar{} \\
                       (S_n,\emptyset) &\fract{\pi}{\longrightarrow} & (N,\emptyset) 
\end{matrix} 
\end{equation}
where $(N,\emptyset)$ is a normal projective algebraic variety $N$ whose algebraic singularities are finite cyclic quotient singularities. So $(N,\emptyset)$ is itself an orbifold. The vertical maps are set theoretically the identity; however, in the orbifold category they are non-trivial Galois coverings with trivial Galois group. This means that there is a non-trivial Galois covering on the level of the local uniformizing neighborhoods with a non-trivial local Galois group. Recall that a Riemannian metric and hence a K\"ahler metric (form) on a complex orbifold is a sequence of K\"ahler metrics on the local uniformizing neighborhoods that are invariant under the local uniformizing groups and that patch together by the injection maps (cf \cite{BG05}, Definition 4.2.11).


We are now ready for the generalization of Theorem 3.8 of \cite{BoTo14a} to the case that $M$ is quasi-regular.
 
\begin{theorem}\label{orbquot}
Let $(M,\cals)$ be a quasi-regular Sasakian structure and $M_{\bfl,\bfw}=M\star_\bfl S^3_\bfw$ the $S^3_\bfw$ join of $M$. Then
the quotient of $M_{\bfl,\bfw}$ by the $S^1$ action generated by a quasi-regular Reeb vector field $\xi_\bfv\in\gt^+_\bfw$ is the orbifold $(S_n,\grD_\bfm+\pi^{-1}(\grD_N))$ where $S_n$ is the total space of the projective orbibundle 
$$(S_n,\grD_\bfm+\pi^{-1}(\grD_N))\fract{\pi^{orb}}{\longrightarrow} (N,\grD_N)$$ 
with generic fibers $\bbc\bbp^1(\bfv)/\bbz_m$ where 
$$m=\frac{l^\infty}{s},\quad n=l^0\frac{(w^0v^\infty -w^\infty v^0)}{s},\quad s=\gcd(l^\infty,|w^0v^\infty-w^\infty v^0 |)$$ 
and $\grD_\bfm$ is a branch divisor consisting of the zero $D^0$ and infinity $D^\infty$ sections of $\bbp(\BOne\oplus L_n)$ with ramification indices $\bfm=(m^0,m^\infty)=m(v^0,v^\infty)$.
Moreover, $c_1^{orb}(L_n)=n [\omega_N] = nc_1^{orb}(M/N)$. 
\end{theorem}

\begin{proof}
The proof given in Sections 3.3-3.5 of \cite{BoTo14a} works equally well when the Sasakian structure on $M$ is quasi-regular with only minor changes. The essential difference is that residual circle action denoted by $S^1_\theta/\bbz_{l_2}$ in \cite{BoTo14a} is now only locally free. As in \cite{BoTo14a} the commutative diagram \eqref{s2comdia} gives the commutative diagram
\begin{equation}\label{comdia2}
\begin{matrix}  M\times L(l^\infty;l^0w^0,l^0w^\infty) &&& \\
                          &\searrow\pi_L && \\
                          \decdnar{\pi_{\bbt^2}} && M_{\bfl,\bfw} &\\
                          &\swarrow\pi_\bfv && \\
                         M_{\bfl,\bfw}/S^1_\phi &&& 
\end{matrix}
\end{equation}
where $L(l^\infty;l^0w^0,l^0w^\infty)$ is a lens space. The action of $\bbt^2=S^1_\phi\times (S^1_\theta/\bbz_{l^\infty})$ on $M\times L(l^\infty;l^0w^0,l^0w^\infty)$ is given by
\begin{equation}\label{T2action3}
(u;z_1,z_2)\mapsto (\cala(\theta,u);[e^{i(v^0\phi-\frac{l^0w^0}{l^\infty}\theta)}z_1,e^{i(v^\infty\phi-\frac{l^0w^\infty}{l^\infty}\theta)}z_2]),
\end{equation}
where now the action $\cala$ on $M$ is only locally free.
This has two main consequences. First, $M_{\bfl,\bfw}$ is a Sasaki orbifold unless Equation \eqref{adcond} holds in which case $M_{\bfl,\bfw}$ will be a smooth Sasaki manifold. Second, and more importantly, the quotient $M_{\bfl,\bfw}/S^1_\theta/\bbz_{l_2}$ can have an additional branch divisor which is the inverse image $\pi^{-1}(\grD_N)$. This gives branch divisors on $(S_n,\grD_{N,\bfm})$ as 
\begin{equation}\label{brdiv}
\grD_{N,\bfm}=(1-\frac{1}{m^0})D^0+(1-\frac{1}{m^\infty})D^\infty +\sum_{j=1}^k\bigl(1-\frac{1}{m_j}\bigr)\pi^{-1}(D_j).
\end{equation}
As is well understood this shifts the orbifold first Chern class. Other than this, the argument is the same. So the result follows as in \cite{BoTo14a}. 
\end{proof}

\begin{remark} Note that for $m$ and $n$ as in Theorem \ref{orbquot} we have
$\gcd(m,n) =1$. This is due to the fact that $\gcd(l^0,l^\infty)=1$,  any factor of $m$ is a factor of $l^\infty$, and, by definition of $s$, $\gcd(m,\frac{w^0 v^\infty - w^\infty v^0}{s}) = 1$.
\end{remark}

Generally we have

\begin{lemma}\label{joinorderlem}
Let $M_{\bfl,\bfw}$ be the $S^3_\bfw$-join with a quasiregular Sasaki manifold $M\longrightarrow N$ of order $\Upsilon_N$. Let $\cals_\bfv=(\xi_\bfv,\eta_\bfv,\Phi,g_\bfv)$ be a quasiregular Sasakian structure in the $\bfw$ cone $\gt^+_\bfw$. The order $\Upsilon$ of $\cals_\bfv$ is the product $mv^0v^\infty\Upsilon_N$.
\end{lemma}

\begin{proof}
The local uniformizing groups of the orbifold are just the isotropy subgroups of the $\bbt^2$ action \eqref{T2action3}, and by definition the order of an orbifold is the least common multiple of the orders of the isotropy groups. 
The orbifold in question is the quotient of $M\times L(l^\infty;l^0w^0,l^0w^\infty)$ by the $\bbt^2$ action \eqref{T2action3} which is free on the dense open subset 
$$\calz=\Bigl(M\setminus \pi_M^{-1}\circ\pi^{-1}(\grD_N)\Bigr)\times \Bigl(L(l^\infty;l^0w^0,l^0w^\infty)\setminus \{[z_1=0]\cup [z_2=0]\}\Bigr)$$
where $|z_1|^2+|z_2|^2=1$ is understood. Let $x\in N$ be a point with isotropy group $\bbz_r$ and $y\in [z_2=0]$. Then the isotropy group at $(x,y)$ is $\bbz_r\times \bbz_{m^0}$. Similarly, if $y\in [z_1=0]$ the isotropy group at $(x,y)$ is $\bbz_r\times \bbz_{m^\infty}$. Running through the orbifold singular set of $N$ gives the order $\Upsilon_N\lcm(m^0,m^\infty)= \Upsilon_Nmv^0v^\infty$.                                                                            
\end{proof}

\subsection{The Transverse K\"ahler Structure}\label{TKS}
It turns out that as in Section 6.1 of \cite{BoTo14a} we can lift the K\"ahler data on the orbifolds $(S_n,\grD_{\bfm,N})$ to the transverse K\"ahler data on a quasi-regular Sasaki manifold. Since quasi-regular Reeb fields are dense in $\gt^+_\bfw$ we can extend this transverse data to the entire $\bfw$-cone and as in \cite{BoTo14a} \eqref{gT2} gives a smooth family of Sasakian structures on the join $M_{\bfl,\bfw}$. In the following we shall describe the details.

\subsubsection{Generalized orbifold Calabi construction}
The  generalized Calabi construction was presented in \cite[Section
 2.5]{ACGT04} and further discussed in \cite[Section 2.3]{ACGT11}. Here we make a modest generalization, allowing for the base space to be a normal projective algebraic variety with cyclic orbifold singularities. We first discussed this type of generalization in \cite{BoTo18c} for the purpose of constructing certain explicit Sasaki-Einstein metrics on a class of $7$-manifolds, but here we will not restrict the base to be a Hirzebruch orbifold. 
  
The ingredients in the construction are as follows:
\begin{itemize}
\item 
$[$Base$]$ A K\"ahler orbifold $(N,\grD_N)$ equipped with a 
K\"ahler orbifold metric, $g_{base}$,  whose K\"ahler form, $\omega_{base}=2\pi \omega_N$, satisfies that
$\left[ \omega_N \right]$ is a primitive orbifold K\"ahler class.
\item  
$[$Fiber$]$ A  weighted projective line $(\bbc \bbp ^1_{m^0,m^\infty} = \bbc \bbp ^1_{v^0,v^\infty}/\bbz_{m},g_{\bfm},\omega_{\bfm})$ with orbifold K\"ahler structure $(g_{\bfm},\omega_{\bfm})$ and rational Delzant polytope $[-1,1] 
\subseteq  \bbr^*$ and momentum map $\gz\colon \bbc \bbp ^1_{m^0,m^\infty} \rightarrow [-1,1]$. 
Here $(m^0,m^\infty)=m(v^0,v^\infty)$ and $v^0,v^\infty$ are coprime.
\item A principal $S^1$ orbi-bundle, $P_n \rightarrow (N,\grD_N)$, with a principal
 connection of curvature $n\,\omega_{base}\in \Omega^{1,1}((N,\grD_N),
\bbr)$, where $S^1$ acts  on $\bbc \bbp ^1_{m^0,m^\infty}$, $n\in {\bbz} \setminus \{0\}$, and $\gcd(n,m)=1$. Note that $n \in \,span_\bbz \{ v^0,v^\infty\}$ (since $v^0,v^\infty$ are coprime), so $mn \in \,span_\bbz \{ m^0,m^\infty\}$. 
\item A constant $0< |r| <1$ with the same sign as $n$ 
[ensuring that the $(1,1)$-form $(1/r + \gz)n\,\omega_{base}$ is positive for $\gz \in [-1,1]$].
\end{itemize}

From this data we may define the orbifold
\[
(S_n, \grD_{{\bfm}, N})=P_n\times_{S^1}\bbc \bbp ^1_{m^0,m^\infty} = \mathring{M}\times_{\bbc^*}\bbc \bbp ^1_{m^0,m^\infty}  \rightarrow (N,\grD_N),
\]
where $ \mathring{S_n} = P_n \times_{S^1}\left( \gz^{-1}(-1,1)\right)$. Since the curvature $2$-form
of $P_n$ has type $(1,1)$, $\mathring{S_n}$ is a holomorphic principal
$\bbc^*$ bundle with connection $\theta \in
\Omega^1(\mathring{S_n},\bbr)$ and $(S_n, \grD_{{\bfm}, N})$ is a complex orbifold.

On $\mathring{S_n}$ we define K\"ahler structures of the form
\begin{equation}\label{compatiblemetric}
\begin{aligned}
g &=(1/r + \gz)n\,(\pi^{orb})^*g_{base} + \frac{1}{\Theta(\gz)}d\gz^2 + \Theta(\gz) \theta^2\\
\omega &=  (1/r + \gz)n\,(\pi^{orb})^*\omega_{base} + d\gz \wedge \theta\\
d\theta &=n\,(\pi^{orb})^*\omega_{base},
\end{aligned}
\end{equation}
where $\frac{1}{\Theta(\gz)} = \frac{d^2 U}{d \gz^2}$ and $U$ is the symplectic
potential \cite{Gui94b} of the chosen toric K\"ahler structure $g_{\bfm}$ on $\bbc \bbp ^1_{m^0,m^\infty}$.

The {\em generalized Calabi construction} arises from seeing
\eqref{compatiblemetric} as a blueprint for the construction of various orbifold
K\"ahler metrics on $(S_n, \grD_{{\bfm}, N})$ by choosing various smooth functions $\Theta(\gz)$ on $(-1,1)$ satisfying that
\begin{itemize}
\item \textup{[boundary values]} 
the following endpoint conditions are satisfied
\begin{equation}\label{eq:toricboundary}
\Theta(\pm 1) =0\qquad {\rm and}\qquad \Theta\,'(-1) =2/m^\infty\qquad {\rm and}\qquad \Theta\,'(1)=-2/m^0;
\end{equation}
\item \textup{[positivity]} 
the function $\Theta(\gz)$ is positive for $\gz\in(-1,1)$.
\end{itemize}
Metrics
constructed this way are called {\em compatible K\"ahler metrics} with {\em compatible K\"ahler classes} parametrized by
$r$. If $g_{base}$ has constant curvature we call the compatible metrics/classes {\em admissible}. Moving forward we will assume to be in that case.

Note that due to \cite{ApCaGa06} (see also (46) of \cite{BoTo14a}), we have that if
$g_{base}= 2\pi g_N$ has complex dimension $d_N$ and constant scalar curvature $2d_N A_N$, then on $\mathring{S_n}$ the scalar curvature of the admissible metric in \eqref{compatiblemetric} is given by
\begin{equation}\label{scalarcurv}
Scal = \frac{2d_N A_Nr}{n(1+r \gz)} - \frac{F''(\gz)}{(1+r\gz)^{d_N}},
\end{equation}
where $F(\gz) = (1+r\gz)^{d_N} \Theta(\gz)$.
[Note that ``$A_N$'' plays the same role as ``$A$'' introduced in Section 6.2 of \cite{BoTo14a}.]

\begin{remark}
Note that the Calabi toric metrics in \cite{Leg09} are special cases of the above ansatz. Such orbifold constructions also appear explicitly in Sections  3.5 and 4.4.2  of \cite{ACGL17}.

\end{remark}

On each uniformizing cover of $(S_n, \grD_{{\bfm}, N})$ we have (following Section 1.3 of \cite{ACGT08}) that 
$$\phi^*[\gro] = 2\pi \frac{n}{r}\phi^*(\pi^{orb})^* [\gro_N] + 2\pi( \phi^*PD(D^0)+ \phi^*PD(D^\infty)).$$
Since $\phi^*PD(D^0) - \phi^*PD(D^\infty) = n \phi^*(\pi^{orb})^*[ \gro_N]$, this may be written as
$$\phi^*[\gro] = 2\pi (\frac{n}{r} + n)\phi^*(\pi^{orb})^* [\gro_N] + 4\pi\phi^*PD(D^\infty)$$
or 
\begin{equation}\label{admissibleclass}
[\gro] = 4\pi\left( \frac{n(1+r)}{2r} (\pi^{orb})^* [\gro_N] + PD(D^\infty)\right).
\end{equation}

\subsubsection{The transverse K\"ahler structure}

First note that a primitive orbifold class $[\gro_{N}]$ is obtained from a primitive integer class $[\gro_{N}]_I$, viz
\begin{equation}\label{primclasseqn}
[\gro_{N}]=\frac{[\gro_{N}]_I}{\Upsilon_N}.
\end{equation}

\begin{lemma}\label{Kahclass}
The induced primitive orbifold K\"ahler form $\gro_{n,\bfm,N}$ on the orbifold $(S_n,\grD_{N,\bfm})$ satisfies
$$[\gro_{n,\bfm,N}]=\frac{l^0w^0v^\infty (\pi^{orb})^*[\gro_N]_I+s\Upsilon_N PD(D^\infty)}{\gcd(s\Upsilon_N,w^0v^\infty l^0) mv^0v^\infty\Upsilon_N}$$
where $[\gro_N]_I$ is a primitive integer class.
\end{lemma}

\begin{proof}
The analogue of Diagram (35) in \cite{BoTo14a} is
\begin{equation}\label{t3comdia}
\begin{matrix} &&  M\times L(l^\infty;l^0w^0,l^0w^\infty) && \\
                          &&\decdnar{\pi_L} && \\
                        && M_{\bfl,\bfw}&&\\
                        &&&& \\
                          &\swarrow p_\bfw &&\searrow p_\bfv & \\
                          &&&& \\
                         (N,\grD_N)\times\bbc\bbp^1[\bfw] &&&& (S_n,\grD_{\bfm,N}) \\
                         &&&& \\
                         &pr_1 \searrow &&\swarrow \pi^{orb} & \\
                         && (N,\grD_N) &&
\end{matrix}
\end{equation}
where $p_\bfw,p_\bfv,pr_1,\pi^{orb}$ are the obvious projections. The fibers of the map $\pi^{orb}$ are orbifolds of the form $\bbc\bbp^1[\bfv]/\bbz_m$. As is \cite{BoTo14a} and by using the diagram \eqref{t3comdia} we have that there is a generator $\grg \in H^2(M_{\bfl,\bfw}, \bbz)$ such that
$(pr_1\circ p_\bfw)^* [\gro_N]=l^\infty\grg$ and $(pr_2\circ p_\bfw)^* [\gro_\bfw]=-l^0\grg$. Now (again, similarly to \cite{BoTo14a}), using Lemma \ref{c1orblem} together with
the fact that 
$$c_1^{orb}( (N,\grD_N)\times\bbc\bbp^1[\bfw] )= pr_1^*c_1^{orb}(N,\grD_N) + |\bfw| pr_2^*[\gro_\bfw],$$
and that, as the diagram \eqref{t3comdia} suggests, both pull-backs to $M_{\bfl,\bfw}$ of the orbifold Chern class $c_1^{orb}(N,\grD_N)$ should agree, we arrive at the equation
$$\frac{1}{m^0}p^*_\bfv PD(D^0)+\frac{1}{m^\infty}p^*_\bfv PD(D^\infty) =-l^0|\bfw|\grg.$$
On each uniformizing cover we also have that 
$$\varphi^*PD(D^0)-\varphi^*PD(D^\infty) = n \varphi^*(\pi^{orb}_\bfv)^*[\gro_N]$$ 
and hence using $(\pi^{orb}\circ p_\bfv)^*[\gro_N]=(pr_1\circ p_\bfw)^* [\gro_N]=l^\infty\grg$ we end up with the system

\begin{eqnarray}\label{sectiondiff}
p^*_\bfv PD(D^0)-p^*_\bfv PD(D^\infty)&=&l^\infty n\grg, \\
\frac{1}{m^0}p^*_\bfv PD(D^0)+\frac{1}{m^\infty}p^*_\bfv PD(D^\infty) &=&-l^0|\bfw|\grg.
\end{eqnarray}
As in \cite{BoTo14a} the solution to this system is
\begin{eqnarray}\label{solnsectiondiff}
p^*_\bfv PD(D^\infty)&=&-mw^0v^\infty l^0  \grg \\
p^*_\bfv PD(D^0) &=& -mw^\infty v^0 l^0\grg
\end{eqnarray}

We keep diagram \eqref{t3comdia} in mind.
Now the induced orbifold K\"ahler class on $(S_n,\grD_{N,\bfm})$ is in the kernel of $p_\bfv^*$, and as in the proof of Lemma 3.10 in \cite{BoTo14a} both $\{(\pi^{orb})^*[\gro_N],PD(D^0)\}$ and $\{(\pi^{orb})^*[\gro_N],PD(D^\infty)\}$ span $\ker (p_\bfv\circ\pi_L)^*$. Now we can write a primitive orbifold K\"ahler form $\gro_{n,\bfm,N}$ as a primitive integer K\"ahler form divided by the order $\Upsilon$ of the orbifold. So on $(S_n,\grD_{N,\bfm})$ using Lemma \ref{joinorderlem} we write
$$[\gro_{n,\bfm,N}]=\frac{k_1(\pi^{orb})^*[\gro_N]_I+k_2PD(D^\infty)}{mv^0v^\infty\Upsilon_N}$$
for some positive integers $k_1,k_2$. 
Since $[\gro_{n,\bfm,N}]$ is in the kernel of $p_\bfv^*$ we have
$$k_1\Upsilon_Nl^\infty-k_2mw^0v^\infty l^0=0.$$
This determines $k_1/k_2=mw^0v^\infty l^0/(l^\infty\Upsilon_N)=mw^0v^\infty l^0/(ms\Upsilon_N)=w^0v^\infty l^0/(s\Upsilon_N)$.
Ensuring primitivity then gives the result.
\end{proof}

\begin{lemma}\label{Kahclass-adm}
The induced primitive orbifold K\"ahler form $\gro_{n,\bfm,N}$ on the orbifold $(S_n,\grD_{N,\bfm})$ satisfies
$$[\gro_{n,\bfm,N}]=\frac{s} {4\pi \gcd(s\Upsilon_N,w^0v^\infty l^0) mv^0v^\infty}[\gro],$$
where $[\gro]$ is the admissible K\"ahler class from \eqref{admissibleclass} with $r=\frac{w^0v^\infty-w^\infty v^0}{w^0v^\infty+w^\infty v^0}$.
\end{lemma}

\begin{proof}
From \eqref{admissibleclass} and \eqref{primclasseqn} we have that
$$
\begin{array}{ccl}
[\gro] & = & 4\pi\left( \frac{n(1+r)}{2r} \frac{(\pi^{orb})^*[\gro_{N}]_I}{\Upsilon_N} + PD(D^\infty)\right)\\
\\
&=& \frac{4\pi}{s \Upsilon_N}\left( \frac{s n(1+r)}{2r} (\pi^{orb})^*[\gro_{N}]_I + s \Upsilon_N PD(D^\infty)\right)\\
\\
&=&  \frac{4\pi}{s \Upsilon_N}\left( \frac{l^0(w^0v^\infty-w^\infty v^0)(1+r)}{2r} (\pi^{orb})^*[\gro_{N}]_I + s \Upsilon_N PD(D^\infty)\right).
\end{array}
$$
This is a multiple of $[\gro_{n,\bfm,N}]$ iff $r=\frac{w^0v^\infty-w^\infty v^0}{w^0v^\infty+w^\infty v^0}$ and with this value for $r$, we get
$$
\begin{array}{ccl}
[\gro]  &=&  \frac{4\pi}{s \Upsilon_N}\left( \frac{l^0(w^0v^\infty-w^\infty v^0)(1+r)}{2r} (\pi^{orb})^*[\gro_{N}]_I + s \Upsilon_N PD(D^\infty)\right)\\
\\
&=& \frac{4\pi}{s \Upsilon_N}\left( l^0w^0v^\infty(\pi^{orb})^*[\gro_{N}]_I + s \Upsilon_N PD(D^\infty)\right)\\
\\
&=& \frac{4\pi \gcd(s\Upsilon_N,w^0v^\infty l^0) mv^0v^\infty}{s}\frac{\left( l^0w^0v^\infty(\pi^{orb})^*[\gro_{N}]_I + s \Upsilon_N PD(D^\infty)\right)}{\gcd(s\Upsilon_N,w^0v^\infty l^0) mv^0v^\infty \Upsilon_N}\\
\\
&=&\frac{4\pi \gcd(s\Upsilon_N,w^0v^\infty l^0) mv^0v^\infty}{s} [\gro_{n,\bfm,N}].
\end{array}
$$
The result now follows.
\end{proof}
Note that since $\gro_{n,\bfm,N}$ is the induced primitive orbifold K\"ahler form on $(S_n,\grD_{N,\bfm})$, we could write
$$[\gro_{n,\bfm,N}]= \frac{[\gro_{n,\bfm,N}]_I}{mv^0v^\infty\Upsilon_N},$$
where $$[\gro_{n,\bfm,N}]_I=  \frac{l^0w^0v^\infty (\pi^{orb})^*[\gro_N]_I+s\Upsilon_N PD(D^\infty)}{\gcd(s\Upsilon_N,w^0v^\infty l^0)}.$$
Using this notation we can formulate the equation in Lemma \ref{Kahclass-adm} as
$$[\gro_{n,\bfm,N}]_I=\frac{s \Upsilon_N} {4\pi \gcd(s\Upsilon_N,w^0v^\infty l^0)}[\gro].$$

Lemma \ref{Kahclass-adm} tells us that the class of the induced primitive orbifold K\"ahler form $\gro_{n,\bfm,N}$ on the orbifold $(S_n,\grD_{N,\bfm})$ 
is a constant multiple of the admissible K\"ahler class $[\omega]$ with the prescribed value of $r$. Since
$(\pi^{orb})^*\gro_{n,\bfm,N}$ is in turn a constant multiple of $d\eta_\bfv$,
this means that for the admissible K\"ahler form, we may say that (up to isotopy)
\begin{equation}\label{omega&eta}
p_\bfv^*\omega = b d\eta_\bfv
\end{equation}
for some positive constant $b$. To determine the value of $b$ we can follow the proof of Proposition 6.2 of \cite{BoTo14a}. Since our notation here is slightly different we will summarize the argument in broad strokes:

As in \cite{BoTo14a} we have the moment map of the lifted circle action of the moment map $\gz$, $\gtz:M_{\bfl,\bfw}\ra{1.6} [-1,1]$ and
\begin{equation}\label{dvw}
d\eta_\bfv|_\cald = \frac{w^0v^\infty-w^\infty v^0}{2v^0v^\infty}(\gtz+r^{-1}) d\eta_\bfw|_\cald.
\end{equation}
Moreover, identifying $N$ with the zero section of $S_n$, $\omega|_N = 2\pi n (\gtz+r^{-1})  \omega_N$.
Now comparing coefficients of the pullback of $\omega_N$ on both sides of equation \eqref{omega&eta} and using equation \eqref{dvw} with the commutative diagram \eqref{t3comdia} (giving that the two pull-backs of $\gro_N$ should agree), we get the equation
$$2\pi n= b \frac{w^0v^\infty-w^\infty v^0}{2v^0v^\infty} l^0.$$
Using the value of $n$ from Theorem \ref{orbquot} we arrive at  $b= \frac{4\pi m v^0 v^\infty}{l^\infty}$ and thus we have
that - up to isotopy -
the transverse K\"ahler structure $(g^T,\omega^T)$ satisfies
$$
\omega^T = d\eta_\bfv = \frac{l^\infty}{4\pi} \frac{p_\bfv^*\omega}{mv^0v^\infty}
$$
and 
$$
g^T= \frac{l^\infty}{4\pi} \frac{p_\bfv^*g}{mv^0v^\infty}.
$$

Thus, similarly to Section 6 of \cite{BoTo14a} we can lift the admissible data to $M_{\bfl,\bfw}$ and we have the following theorem:
\begin{theorem}\label{transKahthm}
The transverse K\"ahler metric and K\"ahler form are (possibly up to isotopy) given by
\begin{eqnarray}\label{gT2}
g^T&=&\frac{l^\infty}{4\pi}\Bigl(\frac{2\pi l^0(w^0v^\infty-w^\infty v^0)}{l^\infty v^0v^\infty}(r^{-1}+\gtz)(\pi^{orb}\circ p_\bfv)^*g_N +\frac{d\gtz^2}{\tilde{\Theta}(\gtz)} +\tilde{\Theta}(\gtz)\ttheta^2\Bigr) \notag \\
\gro^T&=&\frac{l^\infty}{4\pi}\Bigl(\frac{2\pi l^0(w^0v^\infty-w^\infty v^0)}{l^\infty v^0v^\infty}(r^{-1}+\gtz)(\pi^{orb}\circ p_\bfv)^*\gro_N +d\gtz\wedge \ttheta\Bigr)
\end{eqnarray}
where 
$$r=\frac{w^0v^\infty-w^\infty v^0}{w^0v^\infty+w^\infty v^0},\quad \tTheta=mv^0v^\infty p_\bfv^*\Theta,\quad  \ttheta=\frac{p_\bfv^*\theta}{mv^0v^\infty},$$ 
and  that $\tTheta$ satisfies the boundary conditions $\tTheta(\pm 1)=0,$ $\tTheta'(-1)=2v^0$ and $\tTheta'(1)=-2v^\infty$. 

The quasi-regular transverse K\"ahler structure extends smoothly to the full $\bfw$ Sasaki cone $\gt^+_\bfw$ and converges to the K\"ahler forms of the reducible Sasakian structure \cite{BHLT16} as $\bfv\rightarrow \bfw$.
\end{theorem}
It is clear from the explicit forms \eqref{gT2} that they extend smoothly to the full $\bfw$ Sasaki cone $\gt^+_\bfw$. And since
$$\frac{l^\infty}{4\pi}\frac{2\pi l^0(w^0v^\infty-w^\infty v^0)}{l^\infty v^0v^\infty}r^{-1}\longrightarrow l^0$$
we see that as $\bfv\rightarrow \bfw$, the Sasakian structure converges to the reducible Sasakian structure defined by the join.
So the transverse K\"ahler structure in Theorem \ref{transKahthm} extends smoothly to the full $\bfw$ Sasaki cone $\gt^+_\bfw$ and converges to the K\"ahler forms of the reducible Sasakian structure \cite{BHLT16} as $\bfv\rightarrow \bfw$.

By combining Lemma \ref{Kahclass-adm} and Theorem \ref{transKahthm}, we have that
$$
\gro^T=d\eta_\bfv=\frac{l^\infty}{4\pi m v^0 v^\infty}\frac{4\pi \gcd(s\Upsilon_N,w^0v^\infty l^0) mv^0v^\infty}{s} p^*_\bfv \gro_{n,\bfm,N},
$$
which gives the corollary below.

\begin{corollary}\label{transKahformlem}
The transverse K\"ahler structure \eqref{gT2} satisfies
\begin{equation}\label{transKah}
\gro^T=d\eta_\bfv=m\gcd(s\Upsilon_N,w^0v^\infty l^0) p^*_\bfv \gro_{n,\bfm,N}.
\end{equation}

\end{corollary}

\begin{remark}\label{primrem2}
We see, as discussed in Remark \ref{primrem}, the quotient orbifold $(S_n,\grD_{\bfm,N})$ of the primitive Sasakian structure on $M_{\bfl,\bfw}$ is not necesarily primitive. 
\end{remark}

\subsection{Extremal, CSC, and KE admissible metrics on $(S_n,\grD_{\bfm,N})$}\label{kahlergeometry}
Completely similar to Sections 5.1-5.3 of \cite{BoTo14a} (which is turn is based on the admissible construction of Apostolov, Calderbank, Gauduchon, and T{\o}nnesen-Friedman \cite{ApCaGa06,ACGT04,ACGT08}) we may now impose geometric conditions on the metric $g$ in \eqref{compatiblemetric}.
Using $F(\gz) = (1+r\gz)^{d_N} \Theta(\gz)$ as in \eqref{scalarcurv} above, we then have
\begin{itemize}
\item $g$ is an {\bf extremal K\"ahler metric} with $Scal= - (\alpha \gz+\beta)$ if and only if
$$F''(\gz) = (1+r \gz)^{d_N-1}\left(\frac{2d_N A_N r}{n} + (\alpha \gz+\beta)(1+r\gz)\right).$$
By integration, there is a unique solution $F_{ext}(\gz)$ (the {\em extremal polynomial}) to this boundary value problem, where the constants $\alpha$ and $\beta$ are such that the endpoint conditions of \eqref{eq:toricboundary} are satisfied. Further, $F_{ext}(\gz)$ corresponds to a genuine orbifold K\"ahler metric on $(S_n,\grD_{\bfm,N})$ if $F_{ext}(\gz)>0$ for $-1<\gz<1$. As in Section 5.2 of \cite{BoTo14a} and the references therein, we can see that this is the case if $A_N \geq 0$. 
\item The extremal polynomial  $F_{ext}(\gz)$ above corresponds to a {\bf CSC K\"ahler metric} with $Scal=-\beta$ if and only if ($\alpha=0$ which is equivalent to)
\begin{equation}\label{candkeqn}
\frac{2A_{N}\left( (1+r)^{d_N+1} - (1-r)^{d_N+1}\right)}{nr(d_N+1)} + \frac{\beta\left( (1+r)^{d_N+2} - (1-r)^{d_N+2}\right)}{r^2(d_N+1)(d_N+2)} + 2c=0,
\end{equation}
where
$$
\beta= \frac{-2 (d_N+1) r \left(m^\infty (1+r)^{d_N} (n+m^0 A_{N})-m^0 (1-r)^{d_N} (-n+m^\infty A_{N})\right)}{nm^0 m^\infty \left((1+r)^{d_N+1}-(1-r)^{d_N+1}\right)}
$$
and\footnote{There is an inconsequencial $r$ missing in the last term in the numerator of Equation (48) in \cite{BoTo14a}}
$$
c=  \frac{2 \left(1-r^2\right)^{d_{N}} (nm^\infty (1-r)+nm^0 (1+r) -  2m^0 m^\infty A_{N})}{nm^0 m^\infty \left((1+r)^{d_{N}+1}-(1-r)^{d_{N}+1}\right)}.
$$

[Note that in this case, $F_{ext}(\gz)>0$ for $-1<\gz<1$ is automatic.]
\item Finally, the extremal polynomial  $F_{ext}(\gz)$ above corresponds to a {\bf KE K\"ahler metric} if and only if
$(N,\omega_N,g_N)$ is KE with index $\cali_N$ (note $A_N=\cali_N$) and
\begin{equation}
\int_{-1}^1 \left((1-\gz)/m^\infty -(1+\gz)/m^0\right) { (1+r \gz)^{d_N}} d\gz = 0
\end{equation}
and
$$2r\cali_N/n = (1+r)/m^\infty + (1-r)/m^0.$$

[Again, in this case, $F_{ext}(\gz)>0$ for $-1<\gz<1$ is automatic.]
\end{itemize}

\subsection{The Orbifold First Chern Class}
We also have the analogue of Lemma 3.9 of \cite{BoTo14a} when $N$ is an orbifold, viz.

\begin{lemma}\label{c1orblem}
For the projective orbibundle 
$$(S_n,\grD_{\bfm,N})\fract{\pi^{orb}}{\longrightarrow} (N,\grD_N)$$ 
we have
$$c_1^{orb}(S_n,\Delta_{\bfm,N}) = (\pi^{orb}_{\bfm})^*c_1^{orb}(N,\grD_N) + \frac{1}{m^0}PD(D^0)+\frac{1}{m^\infty}PD(D^\infty).$$
\end{lemma}

\begin{proof}
For the orbifold $(N,\grD_N)$ the orbifold canonical divisor $K^{orb}_N$ satisfies
$$K^{orb}_N=\varphi^*K_N+\sum_j\bigl(1-\frac{1}{m_j}\bigr)\varphi^*D_j$$
where $\varphi$ is local uniformizing map whose notation we often omit. Taking the Poincar\'e dual and using $c_1(N)=-c_1(K_N)$ then gives
\begin{equation}\label{c1orbN}
c_1^{orb}(N,\grD_N)=\varphi^*c_1(N) +\sum_j\bigl(\frac{1}{m_j}-1\bigr)\varphi^*PD(D_j).
\end{equation}
Now for a projective bundle $\bbp(\BOne\oplus L_n)\fract{\pi}{\longrightarrow} N$ we know from Leray-Hirsch that the cohomology $H^*(\bbp(\BOne\oplus L_n),\bbz)$ is a free module over $H^*(N,\bbz)$ with basis $\{1,x\}$ where $x$ is the cohomology class in $H^2(\bbp(\BOne\oplus L_n),\bbz)$ whose restriction to a fiber $\approx \bbc\bbp^1$ freely generates the cohomology of $\bbc\bbp^1$. It is the first Chern class of the dual of the tautological bundle \cite{BoTu82}. So the vertical bundle $\mathcal{V}(N)$ has first Chern class $PD(D^0)+PD(D^\infty)$ which implies
$$c_1(\bbp(\BOne\oplus L_n))=\pi^*c_1(N)+PD(D^0)+PD(D^\infty).$$
This together with \eqref{c1orbN} gives the result.
\end{proof}



\begin{remark}\label{Reebrem}
A choice of quasi-regular Reeb vector field $\xi_\bfv\in\gt^+_\bfw$ completely determines the quotient orbifold structure $(S_n,\grD_{\bfm,N})$. With the join parameters $\bfl=(l^0,l^\infty)$ and $\bfw=(w^0,w^\infty)$ fixed, the Reeb vector field $\xi_\bfv$ uniquely determines the integers $n,m,s$ by the expressions in Theorem \ref{orbquot}. Note that both $m$ and $s$ are divisors of $l^\infty$.

\end{remark}

\begin{remark}
From the join construction we get a projective orbibundle 
$$\bbp(\BOne\oplus L_n)=(S_n,\grD_{N,\bfm})\fract{\pi}{\longrightarrow} (N,\grD_N)$$ 
with generic fibers of the form $\bbc\bbp^1[\bfv]/\bbz_m$. However, now since $(N,\grD_N)$ is also an orbifold we have singular fibers along the branch divisors $D_j$ of $N$ and the order of the orbifold singularity is given by the ramification index $m_j$. So along $D_j$ the fibers of $\pi$ have the form $\bbc\bbp^1[\bfv]/(\bbz_m\times\bbz_{m_j})$, and along intersections $D_i\cap D_j$ with $i\neq j$ the fibers take the form $\bbc\bbp^1[\bfv]/(\bbz_m\times\bbz_{m_i}\times\bbz_{m_j})$, etc.
\end{remark}


\begin{remark}
In the regular case where $\Upsilon_N=1$ we have 
$$[\gro_{n,\bfm,N}]=\frac{l^0 v^\infty w^0 (\pi^{orb}_\bfv)^*[\gro_N]+s PD(D^\infty)}{\gcd(s,w^0v^\infty l^0) mv^0v^\infty}$$
as expected from Lemma 3.11 in \cite{BoTo14a} and the assumption of primitivity.
\end{remark}

\begin{example}\label{c1join}
If we assume as in \cite{BoTo18c} that $(N,\grD_N)$ is KE with $c_1^{orb}(N,\grD_N)= \cali_N[\gro_N]$ and $l^0, l^\infty$ are chosen such that
\begin{equation}\label{SEjoin}
l^{0}=\frac{\cali_N}{\gcd(w^0+w^\infty,\cali_N)},\qquad   l^\infty=\frac{w^0+w^\infty}{\gcd(w^0+w^\infty,\cali_N)},
\end{equation}
then from Lemma \ref{c1orblem} we calculate that

\begin{eqnarray*}
c_1^{orb}(S_n,\Delta_{\bfm,N}) & = & (\pi^{orb})^*c_1^{orb}(N,\grD_N) + \frac{1}{m^0}PD(D^0)+\frac{1}{m^\infty}PD(D^\infty)\\
\\
&=& \frac{v^\infty(\cali_N m v^0+n) (\pi^{orb})^* [\gro_N] + (v^0+v^\infty)PD(D^\infty)}{mv^0v^\infty}\\
\\
&=& \frac{(l^\infty \cali_N v^0v^\infty+ l^0v^\infty(w^0v^\infty-w^\infty v^0)) (\pi^{orb}_{\bfm})^* [\gro_N] +s (v^0+v^\infty)PD(D^\infty)}{smv^0v^\infty} \\
&=& \frac{((w^0+w^\infty)v^0v^\infty+ v^\infty(w^0v^\infty-w^\infty v^0))l^0 (\pi^{orb})^* [\gro_N] +s (v^0+v^\infty)PD(D^\infty)}{smv^0v^\infty} \\
&=&\frac{(v^0+v^\infty) \bigl(w^0v^\infty l^0 (\pi^{orb})^* [\gro_N]_I +s\Upsilon_N PD(D^\infty)\bigr)}{smv^0v^\infty\Upsilon_N} \\
&=& \Bigl(\frac{v^0+v^\infty}{s}\Bigr)\gcd(l^0w^0v^\infty,s\Upsilon_N)[\gro_{n,\bfm,N}] 
\end{eqnarray*}
which implies that the Fano index of $[\gro_{n,\bfm,N}]$ is
$$\cali_{(S_n,\grD_{\bfm,N})}= \Bigl(\frac{v^0+v^\infty}{s}\Bigr)\gcd(l^0w^0v^\infty,s\Upsilon_N).$$
Now it follows from first principles that the index $\cali_{(S_n,\grD_{\bfm,N})}$ must be an integer. But we can show this directly in our case. First note that since $s$ divides $|\bfw|=w^0+w^\infty$ when $c_1(\cald_{\bfl,\bfw})=0$ and it also divides $(w^0 v^\infty - w^\infty v^0)$, it is a factor of $v^0(w^0+w^\infty) + (w^0 v^\infty - w^\infty v^0) = w^0 (v^0+v^\infty)$. Similarly, it is a factor of $v^\infty(w^0+w^\infty) - (w^0 v^\infty - w^\infty v^0) = w^\infty (v^0+v^\infty)$. So, since $\gcd(w^0,w^\infty)=1$, it must be a factor of $(v^0+v^\infty)$. A similar analysis shows that $\gcd(l^0w^0v^\infty,s)=1$. So the index reduces to
\begin{equation}\label{Fanoindex2}
\cali_{(S_n,\grD_{\bfm,N})}= \Bigl(\frac{v^0+v^\infty}{s}\Bigr)\gcd(l^0w^0v^\infty,\Upsilon_N).
\end{equation}

As a concrete example, using \eqref{Fanoindex2}, the index of the quotient K\"ahler-Einstein orbifold coming from the quasi-regular  Sasaki-Einstein metric on $Y^{13,8}$ is determined to be $\cali_{v_2}=12$ as in agreement with the first example of Example 5.3 of \cite{BoTo18c}. 
Continuing with this example we can now consider the  quasi-regular Sasaki-Einstein metric of the join
$$Y^{13,8}\star_{4,15}S^3_{34,11}$$ and from the data given in \cite{BoTo18c}, \eqref{Fanoindex2} tells us that now the index of the quotient is
$\cali_{v_3}=28$.

There are obstructions to the regularity of Reeb vector fields:

\begin{proposition}\label{reglem}
Let $M_{\bfl,\bfw}$ be a Sasaki join manifold with $l^\infty>2$ and $c_1(\cald_{\bfl,\bfw})=0$. Then the $\bfw$ cone $\gt^+_\bfw$ does not contain a regular Reeb vector field. 
\end{proposition}

\begin{proof}
Assume that $\xi_\bfv\in\gt^+_\bfw$ is a regular Reeb vector field.  Such a $\xi_\bfv$ is regular if and only if $\bfv=(1,1)$ and $m=1$. But this implies that $2<l^\infty=s=\gcd(s,w^0-w^\infty)$ so $s$ divides $w^0-w^\infty$. But also since $c_1(\cald_{\bfl,\bfw})=0$ equations \eqref{SEjoin} hold.  So $s$ divides $|\bfw|=w^0+w^\infty$. But since $s>2$ either $s$ or $\frac{s}{2}$ must divide both $w^0$ and $w^\infty$, and this gives a contradiction since $\gcd(w^0,w^\infty)=1$.
\end{proof}


\end{example}

\subsection{A Categorical Approach}
For a fixed quasiregular Sasaki manifold $(M,\cals)$ we consider the set of CR orbifolds of Sasaki type
\begin{equation}\label{objs}
\{\{(M_{\bfl,\bfw},\cald_{\bfl,\bfw},J_\bfw)\}_{(\bfl,\bfw)\in (\bbz^+)^2\times (\bbz^+)^2}\}
\end{equation}
given by the join construction. Each element $(M_{\bfl,\bfw},\cald_{\bfl,\bfw},J_\bfw)$ consists of a real 2-parameter family of Sasakian structures parameterized by the $\bfw$ Sasaki cone $\gt^+_\bfw$ which is completely characterized by its dense subset of quasiregular Reeb vector fields $\xi_\bfm$. The set of all such Sasakian structures 
$$\{\cals_{\bfl,\bfw,\bfm}\}_{(\bfl,\bfw,\bfm)\in(\bbz^+)^2\times(\bbz^+)^2\times (\bbz^+)^2}$$
are the object set $\calg_{\bfl,\bfw,\bfm}$ of a groupoid whose morphisms are equivalences, that is, orbifold diffeomorphisms $f:M_{\bfl,\bfw}\longrightarrow M_{\bfl',\bfw'}$ 
such that 
\begin{equation}\label{Deqn}
f_*\cald_{\bfl,\bfw}=\cald_{\bfl',\bfw'},\qquad f_*\circ J_\bfw=J_{\bfw'}\circ f_*
\end{equation}
and
$f^*\cals_{\bfl'\bfw',\bfm'}=\cals_{\bfl,\bfw,\bfm}$. Note that the isotropy subgroup of $\calg_1$ at $\cals_{\bfl,\bfw,\bfm}$ is just the Sasaki automorphism group $\gA\gu\gt(\cals_{\bfl,\bfw,\bfm})$.
We work with the full subgroupoid $\calg=(\calg_0,\calg_1)$ whose object set $\calg_0$ satisfies the condition that the components of $\bfl$ and $\bfw$ are relatively prime.

Similarly, consider the set of projective algebraic ruled K\"ahler orbifolds of the form
\begin{equation}\label{kahorbieqn}
\{\calz_{n,\bfm}=(S_n,\grD_\bfm+\pi^{-1}(\grD_N),\gro_{n,\bfm,N})\}_{(n,\bfm)\in \bbz\times(\bbz^+)^2}
\end{equation}
that satisfy $[\gro_{n,\bfm,N}]\in H^2_{orb}(\calz_{n,\bfm},\bbz)$ and $\gcd(m^0,m^\infty,|n|)=1$ where $N$ is fixed. 
These form the object set $\calq_0$ of a groupoid $\calq=\{\calq_{n,\bfm}\}$ whose morphisms are orbifold biholomorphisms intertwining the K\"ahler structures. 

\begin{proposition}\label{catprop}
There is a full functor $F:\calg\longrightarrow \calq$. 
\end{proposition}

\begin{proof}
We construct such a functor $F:\calg\longrightarrow \calq$. Theorems \ref{orbquot} and \ref{transKahthm} shows that each object $\cals_{\bfl,\bfw,\bfm}$ of $\calg$ uniquely describes a projective algebraic ruled K\"ahler orbifold of the form $(S_n,\grD_\bfm+\pi^{-1}(\grD_N),\gro_{n,\bfm,N})$ which satisfy the relations of Theorem \ref{orbquot}. Moreover, morphisms of $\calg$ satisfy Equations \eqref{Deqn}, so they  induce morphisms of $\calq$. Since morphisms of $\calq$ are orbifold biholomorphisms intertwining the K\"ahler structures, they also induce morphisms of $\calg$. Then Proposition 4.22 of \cite{BHLT16} shows that this is invertible on the objects in the image of $F$. That is, $F$ induces a surjective functor
$${\rm Hom}_\calg(A,B)\longrightarrow {\rm Hom}_\calq(F(A),F(B))$$
so it is full. 
\end{proof}

\begin{remark}\label{Nequiv}
We could also consider groupoids which include equivalences at the level of the orbibundle $M\rightarrow N$, but we do not do so here. 
\end{remark}

\begin{remark}
Note that different orbifold structures on the same $S_n$ do not necessarily correspond to the same join $M_{\bfl,\bfw}$. They can correspond to spaces that are not even homotopy equivalent. 
\end{remark}

\begin{example} We give an example of the equivalence of non-trivial equivalences. The involution $(z_1,z_2)\mapsto (z_2,z_1)$ on $S^3$ induces a diffeomorphism $\gri^\perp_\calg:M_{\bfl,\bfw}\longrightarrow M_{\bfl,\bfw^\perp}$ sending $\cals_{\bfl,\bfw,\bfm}$ to $\cals_{\bfl,\bfw^\perp,\bfm^\perp}$ where $\bfw^\perp=(w^\infty,w^0)$ and $\bfm^\perp=(m^\infty,m^0)$. This gives a
nontrivial  involution $\gri^\perp_\calg$ in $\calg_1$.   
Similarly, we have an involution $\gri^\perp_\calq\in\calq_1$ given by the fiber inversion map which interchanges the branch divisors $(D^0,m^0)$ and $(D^\infty,m^\infty)$ and sends $(S_n,\grD_\bfm+\pi^{-1}(\grD_N),\gro_{n,\bfm,N})$ to $(S_{-n},\grD_{\bfm^\perp}+\pi^{-1}(\grD_N),\gro_{-n,\bfm^\perp,N})$. Using \eqref{Deqn} and the relations of Theorem \ref{orbquot} we see that $F(\gri^\perp_\calg)=\gri^\perp_\calq$.
\end{example}




To each element $\cals_{\bfl,\bfw,\bfm}\in \calg_0$ we can associate the first Chern class $c_1(\cald_{\bfl,\bfw})$ of the contact bundle which is an invariant under morphisms $\calg_1$. Similarly, to each element $\calz_{n,\bfm}\in \calq_0$ we associate $c_1^{orb}(S_n,\Delta_{\bfm,N})$ which also is invariant under the morphisms $\calq_1$. Furthermore, the functor $F$ induces a map 
\begin{equation}\label{c1cateqn}
F^*c_1^{orb}(S_n,\Delta_{\bfm,N})= c_1(\cald_{\bfl,\bfw}).
\end{equation}

\begin{example}\label{Ypqex}The $Y^{p,q}$s and Hirzebruch orbifolds.
Here we take $M=S^3$ with its standard SE metric. As shown in Example 6.8 of \cite{BoTo14a} this corresponds to the join $M\star_\bfl S^3_\bfw$ with 
\begin{equation}\label{pqw}
\bfl=(\gcd(p+q,p-q),p),\quad \bfw=\frac{1}{\gcd(p+q,p-q)}\bigl(p+q,p-q\bigr).
\end{equation}
Here we take $p>0$ and $-p<q<p$ where $q=0$ corresponds to $\bfw=(1,1)$. The condition $\gcd(l^0,l^\infty)=1$ corresponds to $\gcd(p,|q|)=1$ when $q\neq 0$ and $p=1$ if $q=0$. Choose a quasiregular Reeb vector field $\xi_\bfv\in\gt^+_\bfw$. Then Theorem \ref{orbquot} gives $F(Y^{p,q},\xi_\bfv)= (S_n,\grD_\bfm)$  where $S_n$ is the ruled surface over $\bbc\bbp^1$ with orbifold fibers $\bbc\bbp^1[\bfw]/\bbz_m$ satisfying $\bfm=m\bfv$ and
$$n=l^0\frac{p(m^\infty-m^0)+q(m^0+m^\infty)}{p},\qquad \gcd(m^0,m^\infty,n)=1.$$
Note that $n$ is an integer since $p=ms$ and $s$ divides $(p+q)v^\infty-(p-q)v^0$. By Proposition \ref{catprop} this gives a left invertible functor from the groupoid $\calg$ whose objects are $(Y^{p,q},\xi_\bfm)$ to the groupoid $\calq$ whose objects are Hirzebruch orbifolds $(S_n,\grD_\bfm)$ with K\"ahler form $c_1^{orb}(S_n,\grD_\bfm)$. The equivalence $\gri^\perp_\calg$ on $\calg$ sending the triple $(p,q,\bfm)$ to $(p,-q,\bfm^\perp)$ maps to the equivalence $\gri^\perp_\calq$ on $\calq$ sending $(n,\bfm)$ to $(-n,\bfm^\perp)$.
\end{example}

\section{Extremal, CSC, and KE admissible metrics on the $S^3_\bfw$-join}
An important property of the $S^3_\bfw$ join construction is that one can apply the admissible constructions from Section \ref{kahlergeometry} to give explicit constructions of existence theorems for extremal and constant scalar curvature Sasaki metrics. This was presented in a recent survey \cite{BoTo14P} as well as our original papers \cite{BoTo13,BoTo14a} when $M$ is regular; however, as discussed above in Theorem \ref{orbquot} the quasi-regular case goes through without much change. 

In particular, the lifted admissible quasi-regular structure in Theorem \ref{transKahthm} 
\begin{itemize}
\item has a lifted extremal polynomial $\tilde{F}_{ext}(\tilde{\gz})= mv^0v^\infty p_\bfv^*(F_{ext}(\gz))$ whose positivity (or not) for $-1<\tilde{\gz}<1$ will determine whether the corresponding expression \eqref{gT2} gives a genuine quasi-regular admissible extremal Sasaki metric within the isotopy class.
\item Setting $b=v^\infty/v^0$, \eqref{gT2} with $\tilde{\Theta}(\tilde{\gz})=\frac{\tilde{F}_{ext}(\tilde{\gz})}{ (1+r\tilde{\gz})^{d_N}}$ gives a genuine  CSC Sasaki metric if and only if $b$ is a root of $f$ where 

\begin{equation}\label{functionf}
\begin{array}{ccl}

f(b) & = &  (w^0)^{2(d_N+1)} b^{2 d_N+3}( A_N l^\infty  + l^0  (d_N + 1)  w^\infty-b (d_N + 1 ) l^0  w^0 )\\

\\

& - & (w^0)^{d_N+2}  (w^\infty)^{d_N} b^{d_N + 3}  (d_N+1)  (A_N ( d_N+1) l^\infty  - l^0  (( d_N+1) w^0 + ( d_N+2) w^\infty))\\

\\

& + & (w^0)^{d_N+1}  (w^\infty)^{d_N+1}  b^{d_N + 2}  (2 A_N d_N (d_N+2) l^\infty  - (d_N+1)(2d_N+3) l^0  (w^0 + w^\infty))\\

\\

& - &  (w^0)^{d_N}  (w^\infty)^{d_N+2}  b^{d_N + 1}(d_N+1) (A_N ( d_N+1) l^\infty  - l^0  ((d_N+2) w^0 + ( d_N+1) w^\infty))\\

\\

& + & (w^\infty)^{2 (d_N + 1)}( b (A_N l^\infty  + l^0  (d_N + 1) w^0)-( d_N + 1 ) l^0  w^\infty).

\end{array}
\end{equation}
\item Assume further that $M$ is an $S^1$ orbibundle over a compact positive K\"ahler-Einstein orbifold $N$ with K\"ahler class $[\gro_N]\in H^2_{orb}(N,\bbz)$ and that the relatively prime positive integers $(l^0,l^\infty)$ are the relative Fano indices given explicitly by 
$$l^{0}=\frac{\cali_N}{\gcd(w^0+w^\infty,\cali_N)},\qquad   l^\infty=\frac{w^0+w^\infty}{\gcd(w^0+w^\infty,\cali_N)},$$ where $\cali_N$ denotes the orbifold Fano index of $N$. Then \eqref{gT2} with $\tilde{\Theta}(\tilde{\gz})=\frac{\tilde{F}_{ext}(\tilde{\gz})}{ (1+r\tilde{\gz})^{d_N}}$ is an
$\eta$-Einstein Sasaki metric if and only if 
\begin{equation}\label{KEintegral}
\int_{-1}^1 \left((1-b)-(1+b)\gz \right)((b+t)+ (b-t)\gz)^{d_N}d\gz = 0,
\end{equation}
where $t=w^\infty/w^0$ is assumed to satisfy $0<t<1$ and (as above) $b=v^\infty/v^0$. Naturally, in this case, $b$ is also a root of $f$ in \eqref{functionf}.
\end{itemize}

\begin{remark}\label{quasiremark} 
The analysis on pages 1053-1054 of \cite{BoTo14a} shows that in the orbifold category there are infinitely many quasi-regular Sasaki-$\eta$-Einstein manifolds of the form $M\star_\bfl S^3_\bfw$. 
The condition that the join $M\star_\bfl S^3_\bfw$ be smooth is given by \eqref{adcond}, and the condition for a Sasaki-Einstein structure is that the components $l^0,l^\infty$ of $\bfl$ are the relative Fano indices given in Theorem \ref{admjoinse}. So such an $M_{\bfl,\bfw}$ with quasiregular Sasakian structure $\cals$ is smooth if and only if 
\begin{equation}\label{orderjoin}
\gcd(\Upsilon_\cals,l^0w^0w^\infty)=1
\end{equation}
where $\Upsilon_\cals=\Upsilon_N$ denotes the order of the Sasaki manifold $(M,\cals)$. 
\end{remark}

In Section \ref{SetaEinsteinpol} we will explore some remarkable polynomials arising from \eqref{KEintegral}.

\subsection{The Apostolov-Calderbank $CR$-twist and irregular solutions}
In a recent paper of V. Apostolov and D. M. J. Calderbank \cite{ApCa18} the notion of a $CR$-twist was introduced. See Definition 4 in \cite{ApCa18}. This casts an illuminating light on the $\bfw$-cone for the join $M\star_\bfl S^3_\bfw$. Indeed, following Section 4.5 of \cite{ApCa18}\footnote{We would also like to thank Vestislav Apostolov for a very helpful e-mail conversation.} and bearing in mind equation \eqref{dvw} we see in broad strokes the following:

For a fixed $CR$-structure, the orbifold-K\"ahler quotient of each quasi-regular ray given by a pair of co-prime $(v^0,v^\infty)$ is actually a $CR$-twist by a positive Killing potential $f$ of the product K\"ahler quotient of the join with respect to the original ray (given by $(w^0,w^\infty)$). Up to homothety, $f = \tilde{\gz} + 1/r$, where we now view $\tilde{\gz}$ as the lift of the moment map of $(\bbc\bbp^1[\bfw], \omega_\bfw)$ to the product. As we know from Section \ref{TKS},
$$1/r = \frac{w^0v^\infty + w^\infty v^0}{w^0v^\infty - w^\infty v^0} = \frac{w^0b + w^\infty }{w^0b - w^\infty }, $$
where $b=v^\infty/v^0$ as above. Allowing $b$ to be irrational, corresponds to letting $1/r$ be irrational and thus the - locally defined - K\"ahler quotients of irregular rays corresponding to $b\in \bbr^+ \setminus \bbq^+$ are also captured as $CR$-twists by $f = \tilde{\gz} + 1/r$.

Following \cite{ApCa18}, the construction of the quasi-regular admissible extremal Sasaki metrics from above may now be seen as arising as follows. By Theorem 1 in \cite{ApCa18} and the fact that all the Reeb vector fields in the $\bfw$-cone commute with $\xi_\bfw$, the Reeb vector field $\xi_\bfv$ determined by $f=\tilde{\gz} + 1/r$ is extremal if and only if $\xi_\bfw$ is so-called $(\xi_\bfv,d_N+3)$-extremal (which in turn means that the quotient product metric is $(f,d_N+3)$-extremal). The latter condition does not hold if we use the canonical Sasaki Extremal structure on $S^3_\bfw$, but one can change the profile function from the one of the original extremal K\"ahler structure $(\omega_\bfw,g_\bfw)$ to one 
($\tilde{A}$ from (17) in \cite{ApCa18})
that does satisfy that the resulting product metric on $N \times \bbc\bbp^1[\bfw]$ is $(f,d_N+3)$-extremal. The only ``catch'', is that this new profile function may not be positive everywhere. Ignoring this worry for a moment, we now have changed the $CR$-structure of the original join but we remain in the same isotopy class (with the usual ``symplectic to complex'' viewpoint shift) and the resulting Sasaki structure at $\xi_\bfv$ looks like \eqref{gT2} with  $\tilde{\Theta}(\tilde{\gz})=\frac{\tilde{F}_{ext}(\tilde{\gz})}{ (1+r\tilde{\gz})^{d_N}}$. Now the positivity condition of the aforementioned profile function (the $\tilde{A}$) is equivalent to the positivity condition of $\tilde{F}_{ext}(\tilde{\gz})$.

One great advantage of keeping the change of the $CR$-structure at the $\xi_\bfw$-ray site and then ``do a twist" to $\xi_\bfv$, is that it becomes very clear why we may extend our constructions of extremal admissible Sasaki metrics to $b\in \bbr^+ \setminus \bbq^+$ (as we did in \cite{BoTo14a}). It simply corresponds to allowing the number $\tilde{b}_0:= 1/r$ to be in $\bbr^+ \setminus \bbq^+$ and constructing $(\tilde{\gz} + \tilde{b}_0,d_N+3)$-extremal quasi-regular K\"ahler metrics on the
product $N \times \bbc\bbp^1[\bfw]$. The profile function $\tilde{A}$ from (17) of \cite{ApCa18}, clearly varies smoothly with $\tilde{b}_0$ and, in turn, so the resulting admissible extremal metric at $\xi_\bfv$, with corresponding  $\tilde{\Theta}(\tilde{\gz})=\frac{\tilde{F}_{ext}(\tilde{\gz})}{ (1+r\tilde{\gz})^{d_N}}$, will vary smoothly with $(v^0,v^\infty)$. 
In particular, when $b=v^\infty/v^0$ and $b$ is a root of $f$ in \eqref{functionf} then for 
$b\in \bbq^+$ we have a quasi-regular Sasaki CSC metric, whereas when $b\in \bbr^+ \setminus \bbq^+$, we have an irregular Sasaki CSC metric. Likewise, if $b$ is a root of $f$ in \eqref{KEintegral}, and the other conditions in the bullet point containing \eqref{KEintegral} are satisfied, then for $b\in \bbq^+$ we have a quasi-regular $\eta$-Einstein Sasaki metric, whereas when $b\in \bbr^+ \setminus \bbq^+$, we have an irregular $\eta$-Einstein Sasaki metric. 

\subsection{Constant Scalar Curvature}
Summarizing our results we see that by following the same arguments as in Section 6.2 of \cite{BoTo14a} we have orbifold versions of Theorems 1.1 through 1.4. of \cite{BoTo14a}:

\begin{theorem}\label{admjoincsc}
Let $M_{\bfl,\bfw}=M\star_\bfl S^3_\bfw$ be the $S^3_\bfw$-join with a quasi-regular Sasaki manifold (orbifold) $M$ which is an $S^1$ orbibundle over a compact K\"ahler orbifold $N$ with constant scalar curvature $s_N$. Then for each vector $\bfw=(w^0,w^\infty)\in \bbz^+\times\bbz^+$ with relatively prime components satisfying $w^0>w^\infty$ there exists a Reeb vector field $\xi_\bfv$ in the 2-dimensional $\bfw$-Sasaki cone on $M_{\bfl,\bfw}$ such that the corresponding ray of Sasakian structures $\cals_a=(a^{-1}\xi_\bfv,a\eta_\bfv,\Phi,g_a)$ has constant scalar curvature. Moreover, if $s_N\geq 0$, then the $\bfw$-Sasaki cone $\gt^+_\bfw$ is exhausted by extremal Sasaki metrics. If $s_N>0$ and $l^\infty$ is sufficiently large then the $\bfw$-cone has at least 3 cscS rays.

\end{theorem}

\begin{theorem}\label{admjoinse}
Let $M_{\bfl,\bfw}=M\star_{\bfl}S^3_\bfw$ be the $S^3_\bfw$-join with a quasi-regular Sasaki manifold (orbifold) $M$ which is an $S^1$ orbibundle over a compact positive K\"ahler-Einstein orbifold $N$ with primitive K\"ahler class $[\gro_N]\in H^2_{orb}(N,\bbz)$. Assume that the relatively prime positive integers $(l^0,l^\infty)$ are the relative Fano indices given explicitly by 
$$l^{0}=\frac{\cali_N}{\gcd(w^0+w^\infty,\cali_N)},\qquad   l^\infty=\frac{w^0+w^\infty}{\gcd(w^0+w^\infty,\cali_N)},$$ where $\cali_N$ denotes the orbifold Fano index of $N$.  Then for each vector $\bfw=(w^0,w^\infty)\in \bbz^+\times\bbz^+$ with relatively prime components satisfying $w^0>w^\infty$ there exists a Reeb vector field $\xi_\bfv$ in the 2-dimensional $\bfw$-Sasaki cone  $\gt^+_\bfw$ on $M_{\bfl,\bfw}$ such that the corresponding Sasakian structure $\cals_\bfv=(\xi_\bfv,\eta_\bfv,\Phi,g)$ is Sasaki-Einstein. Moreover, this ray is the only admissible CSC ray in the Sasaki cone. 
\end{theorem}

\begin{remark}\label{cscrem}
The constant scalar curvature ray in Theorems \ref{admjoincsc} and \ref{admjoinse} can be either quasi-regular or irregular. In the former case the procedure can be iterated.
\end{remark}

In dimension 5 we have
\begin{corollary}\label{5cor}
A cone decomposable compact Sasaki 5-manifold $M^5$ is diffeomorphic to a lens space bundle over a Riemann surface of genus $g$. Moreover, if $g>0$ it admits a cscS ray in its Sasaki cone; whereas, if $g=0$ choosing the unique ray of Sasakian structures with constant $\Phi$ sectional curvature on $M^3$ (where $M^5=M^3\star_\bfl S^3_\bfw$) gives a constant scalar curvature Sasaki metric on $M^5$.
\end{corollary}


\begin{proof}
By definition $M^5$ is the join $M^3\star_\bfl S^3_\bfw$ where $M^3$ is a Sasaki 3-manifold. By \cite{Gei97} $M^3$ is diffeomorphic to one of the three 3-manifolds 
$$\grG\backslash S^3,\qquad  \grG\backslash\calh^3,\qquad \grG\backslash \widetilde{\rm SL}(2,\bbr)$$
where $\widetilde{\rm SL}(2,\bbr)$ is the universal cover of $SL(2,\bbr)$, $\calh^3$ is the Heisenberg group, and $\grG$ is a discrete group of isometries with respect to the Sasaki metric. These are all $S^1$ Seifert bundles over a Riemann surface of genus $g$ where $g=0$ in the first case, $g=1$ in the second, and $g>1$ in the third. By \cite{Bel01} for cases 2 and 3 there is a contact isotopy to a Sasaki metric of constant $\Phi$ sectional curvature. Whereas, if $g=0$ there is a unique ray of constant $\Phi$ sectional curvature in the Sasaki cone. The corollary then follows from Theorem \ref{admjoincsc} and \cite{BoTo13}.
\end{proof}

\begin{remark}\label{nothighdim}
The analogue of this corollary is definitely false in higher dimensions.
\end{remark}

\begin{remark}
It should be noted that the methods of \cite{ApCa18} also has the advantage of being able to go beyond the $\bfw$-cone (within the full Sasaki-cone), for the join $M^3\star_\bfl S^3_\bfw$, whenever the set of Killing potentials of $(N,\omega_N,\g_N)$ is non-empty. This could be useful in, for example, exploring illuminating examples for the open questions in \cite{BHLT19}.
\end{remark}

\section{The S-$\eta$-E Polynomials}\label{SetaEinsteinpol}
We now consider the S-$\eta$-E condition and equation \eqref{KEintegral} more closely. We choose a Sasakian structure $\cals_\bfv$  in the $\bfw$-cone of $M_{\bfl,\bfw}$ with Reeb field $\xi_\bfv$. Following the analysis on pages 1053-1054  \cite{BoTo14a} we define $b=\frac{v^\infty}{v^0}=kt_\bfw$ and $t_\bfw=\frac{w^\infty}{w^0}$ together with the polynomials of degree $d:=d_N$
\begin{equation}\label{p+-}
p^{\pm}(k)=\int_{-1}^1(1\pm \gz)\bigl((k+1)+(k-1)\gz\bigr)^{d}d\gz.
\end{equation}
Then equation \eqref{KEintegral} is equivalent with $b=\frac{p^-(k)}{p^+(k)}$ and $t_\bfw=\frac{p^-(k)}{kp^+(k)}$.
We can easily compute the polynomials $p^{\pm}(k)$ explicitly.

\begin{lemma}\label{pm}
The polynomials $p^\pm(k)$ defined in \eqref{p+-} satisfy
$$p^-(k)=\frac{2^{d+2}}{(d+1)(d+2)}\sum_{j=0}^d(d+1-j)k^j,\qquad p^+(k)=\frac{2^{d+2}}{(d+1)(d+2)}\sum_{j=0}^d(j+1)k^j,$$
and 
\begin{equation}\label{vrateqn}
\frac{p^-(k)}{p^+(k)}=\frac{d+1+dk+\cdots +2k^{d-1}+k^d}{1+2k+\cdots +(d+1)k^d}.
\end{equation}
\end{lemma}

\begin{proof}
Performing the integration in \eqref{p+-} explicitly we find
\begin{eqnarray*}
p^+(k)&=&\frac{2^{d+2}}{(d+1)(d+2)}\frac{(d+1)k^{d+2}-(d+2)k^{d+1}+1}{(k-1)^2}, \\
p^-(k)&=&\frac{2^{d+2}}{(d+1)(d+2)}\frac{k^{d+2}-(d+2)k+d+1}{(k-1)^2},
\end{eqnarray*}
from which the expressions for $p^\pm(k)$ are easily verified.
\end{proof}

We can use transverse scaling to reduce our discussion to the 1 dimensional space of rays $\gR_\bfw$ of $\gt^+_\bfw$ which is uniquely determined by its slope. We can thus ignore the factor $\frac{2^d}{(d+1)(d+2)}$ in front of the sums in $p^\pm(k)$.

Then as in \cite{BoTo14a} we have

\begin{theorem}\label{almostreg}
Consider Sasaki join manifolds $M_{\bfl,\bfw}$ with $w^0>w^\infty$ relatively prime positive integers and assume $c_1(\cald_{\bfl,\bfw})=0$. Then for each pair $(w^0,w^\infty)$ of relatively prime positive integers with $w^0>w^\infty$ there exists a unique S-$\eta$-E ray $\gr_\bfv$ where $\bfv=(v^0,v^\infty)=(p^+(k),p^-(k))$
which corresponds to the unique root $k\in (1,\infty)$ of the polynomial
\begin{equation}\label{SetaEpoly}
\calp_\bfw(k)=w^\infty (d+1) k^{d+1} +\sum_{j=0}^d\bigl(|\bfw|j-w^0(d+1)\bigr)k^j.
\end{equation}
Furthermore, the S-$\eta$-E ray is quasi-regular if and only $k\in (1,\infty)\cap\bbq$.
 

\end{theorem}

\begin{proof}
Given the relatively prime integers $w^0>w^\infty$ it follows as in Section 6.2 of \cite{BoTo14a} that there is a unique S-$\eta$-E ray corresponding to $k\in (1,\infty)$ such that
\begin{equation}\label{tjeqn2}
w^\infty kp^+(k)= w^0p^-(k).
\end{equation}
It follows easily from Lemma \ref{pm} that the solutions of \eqref{tjeqn2} correspond precisely to roots of the polynomial $\calp_\bfw$. This correspondence is that the ray $\gr_\bfv$ with $\bfv=(p^+(k),p^-(k))$ corresponds to the root $k\in(1,\infty)$.
That there exists a root $k\in(1,\infty)$ follows since 
$$\calp_\bfw(1)= -\frac{(d+1)(d+2)}{2}(w^0-w^\infty)<0$$ 
and $\lim_{k\rightarrow \infty}\calp_\bfw(k)=+\infty$. 
Moreover,  this root is unique since roots of $\calp_\bfw$ in $(1,\infty)$ correspond to S-$\eta$-E rays which by \cite{MaSpYau06} are unique. 

Since $\bfv=(v^0,v^\infty)$ parameterizes the $\bfw$ cone $\gt^+_\bfw$ on $M_{\bfl,\bfw}$, we see that a Reeb vector field $\xi_\bfv$ is quasi-regular if and only if $\bfv$ lies on the integral lattice $(\bbz^+)^2\subset (\bbr^+)^2$. Otherwise $\xi_\bfv$ will be irregular. Since we use the convention $w^0>w^\infty$ this translates to: $\xi_\bfv$ is quasiregular if and only if $k\in (1,\infty)\cap\bbq$. 
\end{proof} 

\begin{remark}
From the discussion in Section 3.1 in \cite{BoTo14a} it follows that the Gorenstein condition $c_1(\cald_{\bfl,\bfw})=0$ is equivalent to the conditions at the beginning of Theorem \ref{admjoinse} and thus Theorem \ref{almostreg} should be seen as a more explicit exposition of the Sasaki $\eta$-Einstein solutions implied by Theorem \ref{admjoinse}. It is straightforward to generalize these results to the $\bbq$-Gorenstein case when $c_1(\cald_{\bfl,\bfw})$ is a torsion class.
\end{remark}

\begin{remark}
The unique almost regular ray ($(v^0,v^\infty)=(1,1)$, i.e., $b=1$) is never  S-$\eta$-E (since $\calp_\bfw(1) \neq 0$). 
\end{remark}

\begin{remark}
Note that if $k\in\bbz^+$, $\bfw$ satisfies
\begin{equation}\label{weqn}
\bfw=(w^0,w^\infty)=\bigl(\frac{kp^+(k)}{(\gcd(p^-(k),kp^+(k))},\frac{p^-(k)}{(\gcd(p^-(k),kp^+(k))}\bigr).
\end{equation}
\end{remark}


We can give a slightly different way of treating quasi-regular S-$\eta$-E rays.
Note that if $k=\frac{p}{q}$ is any rational number and we define 
$$F(p,q)= (d+1)q^d+dq^{d-1}p+\cdots +2qp^{d-1}+p^d$$
then we have
\begin{equation}\label{ppmF}
\frac{p^-(k)}{p^+(k)}=\frac{F(p,q)}{F(q,p)}.
\end{equation}
We define the map $\grk:(1,\infty)\cap\bbq\longrightarrow \gR_{\bfw} \approx\bbr^+$ by
\begin{equation}\label{gImap}
\grk(k)= \frac{d+1+dk+\cdots +2k^{d-1}+k^{d}}{1+2k+\cdots +dk^{d-1} +(d+1)k^{d}}= \frac{v^\infty}{v^0}.
\end{equation}
We can identify $(1,\infty)\cap \bbq$ with the set $\calr$ of pairs $(p,q)$ of relatively prime positive integers with $p>q$. We have an injective map $\grk:\calr\longrightarrow \calr$ defined by 
\begin{equation}\label{gImap2}
\grk(p,q)= (\frac{F(q,p)}{C},\frac{F(p,q)}{C})=(v^0,v^\infty)
\end{equation}
where $C=\gcd(F(q,p),F(p,q))$.   
Then we have

\begin{corollary}\label{SEcor}
A quasi-regular ray $\gr_\bfv$ defined by $\bfv\in\calr$ is S-$\eta$-E if and only if $\bfv\in {\rm image}$ of $\grk$ and satisfies the constraint
\begin{equation}\label{wcon}
\frac{w^\infty}{w^0}=\frac{qF(p,q)}{pF(q,p)}=\frac{qv^\infty}{pv^0}.
\end{equation}
\end{corollary}

\section{Relation with K-stability}
The purpose of this section is to show that in a well defined sense the $S^3_\bfw$ join operation preserves certain K-stability properties. For quasi-regular Sasaki metrics such stability properties are equivalent to the stability properties of certain K\"ahler orbifolds which was investigated in detail in \cite{RoTh11}. However, to treat the general Sasaki metrics one needs to work on the affine cone. For any contact manifold $M$ we consider the cone $C(M)=M\times \bbr^+$. Choosing a 1-form $\eta$ in the contact structure of $M$ we form an exact symplectic structure on $C(M)$ by defining $\gro=d(r^2\eta)$ where $r\in\bbr^+$. The pair $(C(M),\gro)$ is called the symplectization of the (strict) contact structure $(M,\eta)$. Using the Liouville vector field $\Psi=r\partial_r$ we define a natural almost complex structure $I$ on $C(M)$ by
$$ IX=\Phi X +\eta(X)\Psi,\qquad  I\Psi=-\xi$$
where $X$ is a vector field on $M$, and the Reeb field $\xi$ is understood to be lifted to $C(M)$. Without further ado we shall identify $M$ with $M\times\{r=1\}\subset C(M)$. By adding the cone point we obtain an affine variety $Y=C(M)\cup\{0\}$ which is invariant under the complexification $\bbt^\bbc$ of the torus action $\bbt^k$. This was done by Collins and Sz\'ekelyhidi \cite{CoSz12} and it is this approach that we follow here. We begin by defining two important functionals, the {\it total volume} and the {\it total transverse scalar curvature}, viz. 
\begin{equation}\label{VS}
\bfV_\xi=\int_Mdv_g,\qquad \bfS_\xi=\int_Ms^Tdv_g
\end{equation}
where $s^T$ denotes the transverse scalar curvature of the Sasakian structure $\cals=(\xi,\eta,\Phi,g)$. We define the {\it Donaldson-Futaki invariant} on the polarized affine cone $(Y,\xi)$ by
\begin{equation}\label{DonFut}
{\rm Fut}(Y,\xi,a) := \frac{\bfV_\xi}{n}  D_{a}\left( \frac{\bfS_\xi}{\bfV_\xi} \right)+ \frac{\bfS_\xi D_{a}\bfV_\xi}{n(n+1)\bfV_\xi}.
\end{equation}
To proceed further we need the definition of a special type of degeneration known as a {\it test configuration} due to Donaldson \cite{Don02} in the K\"ahler case and Collins-Sz\'ekelyhidi \cite{CoSz12} in the Sasaki case:

\begin{definition}\label{testconf}
Let $(Y,\xi)$ be a polarized affine variety with an action of a torus $T^{\bbc^*}$ for which the Lie algebra $\gt$ of a maximal compact subtorus $T^\bbr$ contains the Reeb vector field $\xi$. A {\it $T^{\bbc}$-equivariant test configuration} for $(Y,\xi)$ is given by a set of $k$ $T^{\bbc}$-invariant generators $f_1,\dots, f_k$ of the coordinate ring $\calh$ of $Y$ and $k$ integers $w_1, \ldots, w_k$ (weights). The functions $f_1,\dots, f_k$ are used to embed $Y$ in $\bbc^k$ on which the weights $w_1, \ldots, w_k$ determine a $\bbc^*$ action. By taking the flat limit of the orbits of $Y$ to $0\in \bbc$ we get a family of affine schemes $\mathcal{Y} \longrightarrow \bbc.$ There is then an action of $\bbc^*$ on the `central fiber' $Y_0$, generated by $a \in\gt'$, the Lie algebra of  some torus $T'^\bbc \subset {\rm GL}(k,\bbc)$ containing $T^\bbc$. 
\end{definition}

We now obtain the correct notions of K-stability following \cite{Tia97,Don02,CoSz12,CoSz15}

\begin{definition}\label{stabdef}
We say that the polarized affine variety $(Y,\xi)$ is {\it K-semistable} if for each $T^\bbc$ such that $\xi \in\gt$ the Lie algebra of $T^\bbr$ and any $T^\bbc$-equivariant test configuration we have
\begin{equation}\label{Futinv}
{\rm Fut}(Y,\xi,a)\geq 0
\end{equation}
where $a\in \gt'$ is the infinitesimal generator of the induced $S^1$ action on the central fiber $Y_0$. The polarized variety $(Y,\xi)$ is said to be {\it K-polystable} if equality holds only for the product configuration $\mathcal{Y}=Y\times\bbc$.
\end{definition}

The general Yau-Tian-Donaldson conjecture in the Sasaki case, which says that cscS metrics corresponds to some kind of K-stability on the affine cone, is still open. It has has been affirmed in one direction by Collins and Sz\'ekelyhidi. 

\begin{theorem}[\cite{CoSz12}]\label{CoSzcscthem} 
Let $(M,\cals)$ be a Sasaki manifold of constant scalar curvature. Then its polarized affine cone $(Y,\xi)$ is K-semistable. 
\end{theorem}

Since there is a one to one correspondence between a Sasaki manifold and its polarized affine variety, we can also refer to this as the K-semistability of the corresponding Sasaki manifold. We now have an immediate corollary of Theorems \ref{admjoincsc} and \ref{CoSzcscthem}:

\begin{corollary}\label{stabcor}
Let $M$ be a quasi-regular Sasaki manifold with constant scalar curvature, let $Y_{\bfl,\bfw}$ denote the affine cone of the $S^3_\bfw$ join $M_{\bfl,\bfw}=M\star_\bfl S^3_\bfw$. Then there exists a K-semistable polarization $\xi$ of $Y_{\bfl,\bfw}$, or equivalently a K-semistable Sasakian structure $(\xi,\eta,\Phi,g)$ on $M_{\bfl,\bfw}$.
\end{corollary}

In the $\bbq$-Gorenstein case, i.e. $c_1(Y)$ is a torsion class, it is well known that the existence of Sasaki-Einstein metrics of $M$ is equivalent to the existence of Ricci flat K\"ahler metrics on the affine cone $C(M)=M\times\bbr^+$, and it is convenient to include the ``cone point'' and define $Y=C(M)\cup \{0\}$. As before the pair $(Y,\xi)$ is a polarized affine variety, and we can also assume that $\xi$ belongs to the Lie algebra $\gt$ of a maximal torus $\bbt$. However, in this case we have a nowhere vanishing section $\grO$ of the canonical line bundle $\grL^{n+1,0}(C(M))$ that satisfies
$$\pounds_\xi\grO=i(n+1)m\grO$$
for some $m\in\bbz^+$. Collins and Sz\'ekelyhidi  call this pair $(Y,\xi)$ a {\it normalized Fano cone singularity}, and they restrict their test configurations to those with normal central fibers.

\begin{definition}\label{Kstabdef}
A normalized Fano cone singularity $(Y,\xi)$ is called a {\it $\bbt$-equivariantly K-stable Reeb polarization} if for all test configurations with normal central fiber $(Y_0,\xi)$ equation \eqref{Futinv} holds with equality only if $(Y,\xi)$ is isomorphic to $(Y_0,\xi)$.
\end{definition}

This gives the Sasaki version of the famous Chen-Donaldson-Sun theorem.

\begin{theorem}[\cite{CoSz15}]\label{CoSzsethem} 
Let $(M,\cals)$ be a Sasaki manifold with $c_1(\cald)$ a torsion class. Then $(M,\cals)$ admits a Sasaki-Einstein metric if and only if its polarized affine cone $(Y,\xi)$ is $\bbt$-equivariantly K-stable.
\end{theorem}

Combining this theorem with Theorem \ref{admjoinse} gives

\begin{corollary}\label{joinstable}
The $S^3_\bfw$-join, $M_{\bfl,\bfw}=M\star_{\bfl}S^3_\bfw$, described in Theorem \ref{admjoinse} admits a $\bbt$-equivariantly K-stable Reeb polarization.
\end{corollary}

\begin{remark}\label{QGor}
By relaxing the primitivity assumption in Theorem \ref{admjoinse} this corollary easily generalizes to the $\bbq$-Gorenstein case with $c_1(\cald)$ a torsion class. 
\end{remark}

\begin{remark}
In both Corollaries \ref{stabcor} and \ref{joinstable} the K-(semi)stable Sasakian structures on $M_{\bfl,\bfw}$ could be irregular.
\end{remark}

\section{Applications}
It is the purpose of this section to apply our main Theorems \ref{admjoincsc} and \ref{admjoinse} to give explicit examples of cone decomposable Sasakian structures that admit constant scalar curvature or S-$\eta$-E metrics. As shown by Corollary \ref{5cor} cone decomposability is well understood in dimension 5, so we concentrate here on dimension 7 and higher. We shall only give explicit examples in dimension 7, since the complexity increases greatly with dimension.

\subsection{Toric Examples}
Recall that toric contact structures of Reeb type have families of Sasakian structures \cite{BG00b,Ler02a}. It is easy to see that the $S^3_\bfw$ join of a toric contact structure of Reeb type is a toric contact structure of Reeb type. In dimension 5 all toric contact structures of Reeb type on $S^3$ bundles over $S^2$ can be realized as joins and have been studied in \cite{Leg10,BoPa10}. Examples in dimension 7 can easily be constructed for example $Y^{p,q}\star_\bfl S^3_\bfw$ as in \cite{BoTo18c}. This is a special case of the iterated $S^3$ joins which are treated in \cite{BoTo19b}.

\subsection{Complexity One Examples in Dimension $7$}\label{compl1sect}
Here we consider the $S^3_\bfw$ join with a complexity one Sasaki 5-manifold recently studied in \cite{CoSz15}. It is the link $L_{p,q}$ of the Brieskorn varieties 
\begin{equation}\label{Brieslink}
f_{p,q}=z_0^p+z_1^q+z_2^2+z_3^2=0
\end{equation}
of degree $d=2pq$ with its naturally induced Sasakian structure with Reeb vector field $\xi_0$. We denote this join by $M_{p,q,\bfl,\bfw}$. Topologically the 5-manifolds $L_{p,q}$ are precisely the $k$-fold connected sums $k(S^2\times S^3)$ with $k=\gcd(p,q)-1$ and $k=0$ means the five sphere $S^5$. So from Proposition \ref{elemtopprop} we see that the 7-manifold  $M_{p,q,\bfl,\bfw}$ is simply connected with $\pi_2(M_{p,q,\bfw})=H_2(M_{p,q,\bfl,\bfw},\bbz)=\bbz^k$.

The weight vector of the naturally induced Reeb vector field $\xi_0$ is $(2q,2p,pq,pq)$. The Sasaki cone $\gt^+(L_{p,q})$ of the link $L_{p,q}$ is 2 dimensional and was computed in \cite{CoSz15,BovCo16} using the Hilbert series of the coordinate ring on the polarized affine variety $(Y_{p,q},\xi_0)=(\bbc[z_0,z_1,z_2,z_3]/(f_{p,q}),\xi_0)$. It is given by 
\begin{equation}\label{tanform}
\gt^+(L_{p,q})=\{\xi=a\xi_0+ b\grz~|~a>0,~ -\frac{ad}{2}<b<\frac{ad}{2}\}
\end{equation}
where $\grz\in\gt$ has weight $(1,-1)$ with respect to $(u=z_2+iz_3,v=z_2-iz_3)$ and $0$ elsewhere. 
In all cases there is a unique ray  that minimizes the volume on the planar slice fixed by $a$ which corresponds to $b=0$. Moreover, a Theorem of Collins and Sz\'ekelyhidi \cite{CoSz15} says that this ray is cscS, hence S-$\eta$-E, if and only if $2p>q$ and $2q>p$, or the quadric\footnote{We exclude the case $p=q=2$ in what follows. This gives the homogeneous Einstein metric on $S^2\times S^3$, hence, it is regular and thus treated in \cite{BoTo14a}.} $p=q=2$. The Fano index $\cali_{p,q}$ of this ray is 
\begin{equation}\label{pqFanoind}
\cali_{p,q}=2q+2p+pq+pq-d=2(p+q).
\end{equation}
Thus, as in Equation (9) of \cite{BoTo14a} we have
\begin{lemma}\label{c1orbNpq}
There is a generator $\grg\in H^2(M_{p,q,\bfl,\bfw},\bbz)$ such that 
$$c_1(\cald_{p,q,\bfl,\bfw})=(2l^\infty(p+q)-l^0|\bfw|)\grg.$$
Moreover, the second Stiefel-Whitney class is $w_2(M_{p,q,\bfl,\bfw})=l^0|\bfw|\grg\mod 2$.
\end{lemma}

\begin{remark}\label{typechangerem}
For those contact structures $\cald_{p,q,\bfl,\bfw}$ such that $2(p+q)l^\infty>l^0w^0$ the entire $\bfw$ Sasaki cone $\gt^+_\bfw$ consists of positive Sasakian structures. When this inequality does not hold type changing occurs. See Proposition 4.2 of \cite{BoTo15}.
\end{remark}
 
 We can now apply Theorem \ref{admjoincsc} to give a cscS metric in the $\bfw$ Sasaki cone $\gt^+_\bfw$ of the join $M_{p,q,\bfl,\bfw}$. But first we construct the $S^1$ quotient $N$ of $L_{p,q}$ as a log pair. 

\begin{lemma}\label{quotorb5}
Let $k=\gcd(p,q)-1$, then the quotient orbifold $(N,\grD)$ generated by the Reeb vector field $\xi_0$ is $(N,\grD_{p,q})$ where $N$ is the weighted projective hypersurface  embedded in $\bbc\bbp^3[2,2,k+1,k+1]$ as $z_0^{k+1}+z_1^{k+1}+z_2^2+z_3^2=0$ and $\grD_{p,q}$ is the branch divisor
$$\grD_{p,q}=(1-\frac{1}{p})(z_0=0) +(1-\frac{1}{q})(z_1=0).$$
Moreover, there are algebraic quotient singularities with isotropy $\bbz_{k+1}$ occuring at the intersection of the irreducible components of $\grD_{p,q}$, namely the two points $[0,0,1,\pm i]\in N$, and $\bbz_2$ quotient singularities at the points at the $k+1$ points $[1,e^{\frac{i\pi (2m+1)}{k+1}},0,0]$, $m=0,\ldots,k$.
\end{lemma}

\begin{proof}
There are three cases to consider, $p,q$ both odd, $p,q$ both even, and $p,q$ have different parity. In the last two cases, the weights have a common factor of $4$ and $2$, respectively which should be factored out of the weights and the degree. Once this is done the proof uses standard techniques described in Chapter 4 of \cite{BG05}, in particular Lemma 5.4.10. First we define a new weight vector $\bar{\bfw}$ as described in Definition 4.5.7 of \cite{BG05} which by a result of Delorme (Proposition 4.5.10 in \cite{BG05}) gives an isomorphism $\bbc\bbp^3[\bfw]\approx \bbc\bbp^3[\bar{\bfw}]$ as algebraic varieties. For example if $p$ and $q$ are both odd we have from Definition 4.5.7 that  
$$\bar{\bfw}=(2,2,\gcd(p,q),\gcd(p,q))=(2,2,k+1,k+1)$$
and then Lemma 5.4.10 determines the hypersurface. The other two cases are similar once the factorization described above is done. The second statement then follows since the hypersurface inherits the algebraic quotient singularities of $\bbc\bbp^3[2,2,k+1,k+1]$.
\end{proof}

This gives
 \begin{theorem}\label{pqw7man}
 If $2p>q$ and $2q>p$ then for each of the relatively prime pairs $(l^0,l^\infty)$ and $(w^0,w^\infty)$ there exists a ray of constant scalar curvature Sasakian structures in $\gt^+_\bfw$ on the 7 dimensional Sasaki orbifolds $M_{p,q,\bfl,\bfw}$, and the entire $\bfw$ cone $\gt^+_\bfw$ is exhausted by extremal Sasaki metrics. Moreover, if $l^\infty$ is large enough there are at least 3 cscS metrics in $\gt^+_\bfw$. $M_{p,q,\bfl,\bfw}$ is a smooth manifold if and only if $\gcd(2l^\infty pq,l^0w^0w^\infty)=1$.
 \end{theorem}

\begin{remark}\label{spinrem}
It is interesting to note that the last statement of the Theorem together with Lemma \ref{c1orbNpq} imply that if $M_{p,q,\bfl,\bfw}$ is smooth it must be a spin manifold. All the non-spin cases are orbifolds.
\end{remark}

Excluding the homogeneous quadric $p=q=2$, Collins and Sz\'ekelyhidi \cite{CoSz15} prove that the link $L_{p,q}$ has an SE metric if and only if $2p>q$ and $2q>p$. 
The slice that gives the possible SE metric is determined by $c_1(\calf_\xi)=4[d\eta]_B$.

Applying Theorem \ref{admjoinse} gives SE metrics on the $S^3_\bfw$ joins $M_{p,q,\bfl,\bfw}=L_{p,q}\star_\bfl S^3_\bfw$ where the components of $\bfl$ are the relative Fano indices. Since these indices are fixed we relabel the join as $M_{p,q,\bfw}$. We have from Theorem \ref{admjoinse} and Theorem 1.2 of \cite{CoSz15}

\begin{theorem}\label{comp1join}
If $2q>p$ and $2p>q$ there is an SE metric of complexity one  in the $\bfw$ Sasaki cone $\gt^+_\bfw$ of the join 
$$M_{p,q,\bfw}=k(S^2\times S^3)\star_\bfl S^3_\bfw$$
where 
$$l^0=\frac{2(p+q)}{\gcd(2(p+q),|\bfw|)}, \qquad l^\infty=\frac{|\bfw|}{\gcd(2(p+q),|\bfw|)}.$$

\end{theorem}

The converse of this theorem is not known at this time.

From Proposition \ref{elemtopprop} we see that the 7-manifolds $M_{p,q,\bfw}=k(S^2\times S^3)\star_\bfl S^3_\bfw$ are simply connected with $\pi_2(M_{p,q,\bfw})=\bbz^{k+1}$. 

\subsection{Complexity Two Examples in Dimension $7$}\label{comp2sect}
In this case $\gt^+_\bfw$ is the full Sasaki cone, i.e. the dimension of the Sasaki cone is two. Since any Sasakian structure on $M$ is quasi-regular, there are many cases. Here it suffices to consider a simple example. Consider the Brieskorn-Pham link $L_{k,p}$ defined by
\begin{equation}\label{kplink}
z^k_0+z_1^k+z_2^{k+1}+z_3^p=0
\end{equation}
of degree $pk(k+1)$ and weight vector $\bfw=((k+1)p,(k+1)p,kp,k(k+1))$. The index of the natural quotient is 
\begin{equation}\label{com2index}
\cali_{k,p}=|\bfw|-d=2pk+2p+k-(p-1)k^2.
\end{equation}
These hypersurfaces give infinitely many negative Sasakian structures, finitely many positive Sasakian structures, and one null Sasakian structure. First we only consider $k,p>1$ since if $k$ or $p$ is one we obtain the standard 5 sphere. For simplicity we consider the case that $p$ is relatively prime to both $k$ and $k+1$ in which case $L_{k,p}$ is diffeomorphic to $S^5$ (see Example 10.3.12 in \cite{BG05}). This eliminates the null case which has $(k,p)=(3,12)$. It is easy to see that $S^5$ admits no null (Corollary 10.3.9 in \cite{BG05}) nor indefinite Sasakian structures. We also note that $k=2$ is a special case of the complexity one hypersurfaces considered in Section \ref{compl1sect}. So here we consider $k\geq 3$. It is easy to see that the induced Sasakian structure is negative if 
$$k>3,p>3,\quad k=3, p>12, \quad k\geq 6, p=2,3, \quad \text{or}\quad k=5,p=3$$ 
and positive otherwise. So the only positive Sasakian structures on the hypersurfaces  $L_{k,p}\approx S^5$ with $\gcd(k,p)=\gcd(k+1,p)=1$ are $L_{3,5},L_{3,7},L_{3,11},L_{4,3}$ all of which admit SE metrics, hence cscS metrics, which can be seen from \cite{BGK05,BG05}. All negative Sasakian structures on links of weighted hypersurface singularities admit S-$\eta$-E metrics, hence cscS metrics, by the transverse Aubin-Yau Theorem \cite{ElK}. Summarizing we have

\begin{theorem}\label{kp7man}
For all $k,p$ such that $p$ is relatively prime to $k$ and $k+1$, and relatively prime pairs $(l^0,l^\infty)$ and $(w^0,w^\infty)$, there exists a ray of constant scalar curvature Sasakian structures in $\gt^+_\bfw$ on the 7 dimensional Sasaki orbifolds $M_{k,p,\bfl,\bfw}$, and the entire $\bfw$ cone $\gt^+_\bfw$ is exhausted by extremal Sasaki metrics. Moreover, if $l^\infty$ is large enough there are at least 3 cscS metrics in $\gt^+_\bfw$. $M_{k,p,\bfl,\bfw}$ is a smooth manifold if and only if $\gcd(\lcm(k,k+1,p)l^\infty,w^0w^\infty l^0)=1$.
 \end{theorem}

\begin{remark}\label{spinrem2}
Again as in Remark \ref{spinrem} if $M_{k,p,\bfl,\bfw}$ is smooth it must be a spin manifold. All non-spin cases are orbifolds.
\end{remark}

We now have the analogue of Example \ref{S5ex} and Lemma \ref{quotorb5}
\begin{lemma}\label{quotorb52}
If $\gcd(k,p)=\gcd(k+1,p)=1$, then the quotient orbifold $(N,\grD_{k,p})$ generated by the natural Reeb vector field $\xi_0$ on the link $L_{k,p}$ is the log pair $(\bbc\bbp^2[k,1,1],\grD_{k,p})$ with the branch divisor
$$\grD_{k,p}=(1-\frac{1}{k+1})(z_2=0) +(1-\frac{1}{p})(z_3=0).$$
Moreover, there is an algebraic quotient singularity with isotropy $\bbz_{k}$ occuring at the intersection of the irreducible components of $\grD_{k,p}$, namely the points $[1,e^{\frac{i\pi (2m+1)}{k}},0,0]$, $m=0,\ldots,k-1$.
\end{lemma}

\begin{proof}
The proof is similar to that of Lemma \ref{quotorb5} and left as an exercise.
\end{proof}

Note that the order of the link $L_{k,p}$ is $\lcm(k,k+1,p)$.


     %



\subsection{Topology of the Joins}
Clearly, we can apply our Theorems \ref{admjoincsc} and \ref{admjoinse} to obtain many new cscS and SE metrics. A particularly interesting case is when $M$ is a Sasaki homotopy sphere as studied, for example, in \cite{BGN03b,BGK05,GhKo05}. It is well known that there are precisely 28 oriented diffeomorphism types of homotopy spheres in dimension 7 and these all bound a parallelizable manifold. If we take the $S^3_\bfw$ join of these we obtain a simply connected Sasaki 9-manifold with $\pi_2=\bbz$ that bounds a parallelizable manifold. What are the possible diffeomorphism types? There are only 2 diffeomorphism types of homotopy spheres that bound a parallelizable manifold in dimension 9. Similarly there are many more oriented diffeomorphism types in dimension 11 than in dimension 9. Although it is beyond the scope of the present work, it would be very interesting to understand how the join operation relates to these diffeomorphism types.

We have already some elementary results in Proposition \ref{elemtopprop}. Here we adopt the method described in \cite{BoTo14a} to the general quasi-regular situation. The fibration \eqref{S3joineqn} and Diagram (\ref{s2comdia}) together with the torus bundle with total space $M\times S^3_\bfw$ gives the commutative diagram of fibrations
\begin{equation}\label{orbifibrationexactseq}
\begin{matrix}M\times S^3_\bfw &\longrightarrow &M_{\bfl,\bfw}&\longrightarrow
&\mathsf{B}S^1 \\
\decdnar{=}&&\decdnar{\pi}&&\decdnar{\psi}\\
M\times S^3_\bfw&\longrightarrow & \mathsf{B}(N,\grD)\times\mathsf{B}\bbc\bbp^1[\bfw]&\longrightarrow
&\mathsf{B}S^1\times \mathsf{B}S^1\, 
\end{matrix} \qquad \qquad
\end{equation}
where $\mathsf{B}G$ is the classifying space of a group $G$ or Haefliger's classifying space \cite{Hae84} of an orbifold if $G$ is an orbifold. The map $\psi$ of Diagram (\ref{orbifibrationexactseq}) is that induced by the inclusion $e^{i\theta}\mapsto (e^{il^\infty\theta},e^{-il^0\theta})$. Note that the lower fibration is a product of fibrations. The orbifold cohomology of $\bbc\bbp^1[\bfw]$ was computed in \cite{BoTo14a}
$$H^r_{orb}(\bbc\bbp^1[\bfw],\bbz)=H^r( \mathsf{B}\bbc\bbp^1[\bfw],\bbz)= \begin{cases}
                    \bbz &\text{for $r=0,2$,}\\                  
                    \bbz_{w^0w^\infty} &\text{for $r>2$ even,}\\
                     0 &\text{for $r$ odd.}
                     \end{cases}$$      
The method then consists of computing the differentials of the fibration
\begin{equation}\label{Mfib}
M\longrightarrow \mathsf{B}(N,\grD)\longrightarrow  \mathsf{B}S^1
\end{equation}
and using the commutative diagram \eqref{orbifibrationexactseq} to obtain the cohomology ring of $M_{\bfl,\bfw}$. This method dates back to \cite{WaZi90} for the manifold case and was generalized to the $S^3_\bfw$-join for $M=S^{2n+1}$ with its standard constant curvature metric in \cite{BoTo14a}, and in \cite{BoTo18c} when $M$ is a $Y^{p,q}$ with a quasi-regular SE metric. In both cases the torsion group $\bbz_{w^0w^\infty(l^0)^2}$ occurs in $H^4(M_{\bfl,\bfw},\bbz)$. 

\begin{proposition}\label{cohjoin}
Let $M_{\bfl,\bfw}=M\star_\bfl S^3_\bfw$ and assume that $M$ has dimension greater than 3, is simply connected and that $b_3(M)=0$. Then $H^4(M_{\bfl,\bfw},\bbz)$ contains the torsion group $\bbz_{w^0w^\infty(l^0)^2}$.
\end{proposition}

\begin{proof}
Consider the differentials in the Leray-Serre spectral sequence of the bottom fibration of \eqref{orbifibrationexactseq} which is a product. For the $S^3$ fibration we have $d_4(\gra)=w^0w^\infty s_2^2$ where $\gra$ is the orientation class of $S^3$ and $s_2$ is the generating class of the second $\mathsf{B}S^1$. Now for the $M$ fibration Proposition \ref{elemtopprop}, $b_3(M)=0$ and the Universal Coeffients Theorem imply that $H^3(M,\bbz)=0$. It follows that in the spectral sequence of the fibration \eqref{Mfib} $s_1^2$ survives to $E_\infty$. By naturality the differentials in the bottom product fibrations of Diagram \eqref{orbifibrationexactseq} pullback to the differentials of the top fibration, and since the diagram commutes we have $d_4(\gra)=w^0w^\infty(l^0)^2s^2$. This class survives to $E_\infty$ of the top fibration, thus, producing the torsion group in $H^4$.
\end{proof}

Further details depend on the differentials in the spectral sequence of the fibration \eqref{Mfib}.

\begin{example}\label{S5ex}
A particular case of interest in Lemma \ref{quotorb5} is when $\gcd(p,q)=1$ or equivalently $k=0$ which gives $M=S^5$. However, we only need the manifold $M$ to be homeomorphic to $S^5$. In what follows we assume that $M_{\bfl,\bfw}=S^5\star_\bfl S^3_\bfw$. We have


\begin{lemma}\label{7topQ}
$H^3(M_{\bfl,\bfw},\bbz)=0$ and $M_{\bfl,\bfw}$ has the rational cohomology ring of $S^2\times S^5$.
\end{lemma}

\begin{proof}
From the Leray-Serre spectral sequence of the fibration 
$$S^3\times S^5\longrightarrow M_{\bfl,\bfw}\longrightarrow \mathsf{B}S^1$$
we see that $E^{r,s}_2=E^{r,s}_4$. We let $\gra$ be the orientation class of $S^3$, then, over $\bbq$ , there are two possibilities for $d_4$, either $d_4(\gra)=s^2$ or $d_4(\gra)=0$ where $s$ is the 2-class of $\mathsf{B}S^1$. If the latter holds then the 3-class $\gra$ will survive to $E_\infty$, and there is only one such class. But by the Leray-Serre Theorem this converges to the cohomology of $M_{\bfl,\bfw}$ which has a Sasaki metric. This implies that the rank of $H^3(M_{\bfl,\bfw},\bbq)$ must be even which gives a contradiction. Thus, we must have $d_4(\gra)=s^2$, and this implies that $H^3(M_{\bfl,\bfw},\bbq)=0$. But from Proposition \ref{elemtopprop} we have $H_2(M_{\bfl,\bfw},\bbz)=\bbz$, so $H^3$ vanishes over $\bbz$ as well by the Universal Coefficients Theorem. 
Then by naturality we have $d_4(\gra\otimes s^j)=s^{j+1}$. So for $E^{2r,0}$ only $E^{2,0}$ survives to $E_\infty$. The only classes that survive to $E^\infty$ are the 2-class $s$ on the base, the 5-class $\grb$ from the fiber, and the 7-class $\grb\otimes s$.
\end{proof}

We now have
\begin{proposition}\label{mainpropexam}
Consider the join $M_{\bfl,\bfw}=S^5\star_\bfl S^3_\bfw$ and assume also that $M_{\bfl,\bfw}$ is a smooth manifold. Then the cohomology ring of $M_{\bfl,\bfw}$ is
$$\bbz[x,y]/(w^0w^\infty (l^0)^2x^2,x^3,x^2y,y^2)$$
where $x$ is a 2-class and $y$ is a 5-class.
\end{proposition}

\begin{proof}
As with Theorem 4.5 of \cite{BoTo14a} we use naturality applied to the commutative diagram \eqref{orbifibrationexactseq}. In fact the result agrees with Theorem 4.5 although no assumption of regularity of the Reeb vector on $S^5$ is made. However, here we need to impose the smoothness condition; whereas, it is automatic in Theorem 4.5 of \cite{BoTo14a}. As in Theorem \ref{pqw7man} $M_{\bfl,\bfw}$ is smooth if and only if $\gcd(2pql^\infty,w^0w^\infty l^0)=1$. Now we know from Proposition \ref{elemtopprop} and Poincar\'e duality that $H^6(M_{\bfl,\bfw},\bbz)=0$ and $H^5(M_{\bfl,\bfw},\bbz)=\bbz$. Moreover, from Lemma \ref{7topQ} $H^3(M_{\bfl,\bfw},\bbz)=0$ and from Proposition \ref{cohjoin} $H^4(M_{\bfl,\bfw},\bbz)$ contains the torsion group $\bbz_{w^0w^\infty(l^0)^2}$. One easily sees from the Leray-Serre spectral sequence of the upper fibration in Diagram \ref{orbifibrationexactseq} that $H^4(M_{\bfl,\bfw},\bbz)=\bbz_{w^0w^\infty(l^0)^2}$ from which the cohomology ring easily follows.
\end{proof}

\end{example}

\begin{remark}\label{highsphrem}
This proposition generalizes to higher dimensional spheres as long as the smoothness condition of the join holds. That is, if $M_{\bfl,\bfw}=S^{2r+1}\star_\bfl S^3_\bfw$ is a smooth manifold, then its cohomology ring is
$$\bbz[x,y]/(w^0w^\infty (l^0)^2x^2,x^{r+1},x^2y,y^2)$$
where $x$ is a 2-class, and $y$ is a $(2r+1)$-class.
\end{remark}

Proposition \ref{mainpropexam} implies that the $S^3_\bfw$ join with any 5-sphere has the same cohomology ring. In particular, we apply this to the standard toric join, the complexity one links $L_{p,q}$ with $\gcd(p,q)=1$ and the complexity two links $L_{k,p}$ with $\gcd(p,k)=\gcd(p,k+1)=1$. We are now ready to apply Sullivan's rational homotopy theory \cite{Sul77}. As described in \cite{WaZi90} this implies that if a collection of simply connected closed manifolds have the same cohomology ring, the same rational Pontrjagin classes, and if the so-called minimal model is a formal consequence of the rational cohomology ring, then the collection has only finitely many diffeomorphism types. Let $\calc_{p,q,\bfl,\bfw}$ be the collection of smooth $S^3_\bfw$-joins $L_{p,q}\star_\bfl S^3_\bfw$ of Brieskorn links with $w^0w^\infty(l^0)^2$ fixed and satisfy Equation \eqref{Brieslink} with $\gcd(p,q)=1$. Similarly let $\calc_{k,p,\bfl,\bfw}$ be the collection of smooth $S^3_\bfw$-joins $L_{k,p}\star_\bfl S^3_\bfw$ of Brieskorn links with $w^0w^\infty(l^0)^2$ fixed and satisfy Equation \eqref{kplink} with $\gcd(p,k)=\gcd(p,k+1)=1$. These collections have the cardinality of $\bbn$. Then we have

\begin{proposition}\label{rathomprop}
There are finitely many diffeomorphism types within the collection $\calc_{p,q,\bfl,\bfw}$. There are finitely many diffeomorphism types within the collection $\calc_{k,p,\bfl,\bfw}$ with fixed $k$. 
\end{proposition}
 
\begin{proof}
Since $w^0w^\infty(l^0)^2$ is a fixed odd number within each collection, the cohomology rings coincide. So according to Sullivan's theory we need only check that the first Pontrjagin class is fixed within each collection, and that the join is formal in each case. By Lemma \ref{7topQ} all the elements of each collection have the rational cohomology ring of $S^2\times S^5$, so they are formal. To compute the Pontrjagin class we make note of the map $\pi$ in the Commutative Diagram \eqref{orbifibrationexactseq}. Since the vertical bundle is trivial, $p_1(M_{\bfl,\bfw})$ is the pullback of $p_1(\mathsf{B}(N,\grD)\times\mathsf{B}\bbc\bbp^1[\bfw])$. This is well defined since the underlying topological space $(N,\grD)\times \bbc\bbp^1$ is a combinatorial manifold determined up to combinatorial equivalence \cite{RoSv57,Tho58}. Then it follows from Lemma \ref{quotorb5} that for the collection $\calc_{p,q,\bfl,\bfw}$ the underlying topological space is independent of $p,q,l^\infty$, so the first statement follows. In the case of the collection $\calc_{k,p,\bfl,\bfw}$ Lemma \ref{quotorb52} tells us that the quotient singularity of the underlying topological space depends on $k$, so we should fix $k$ for the underlying spaces to be combinatorially equivalent.
\end{proof}

From Lemma \ref{c1orbNpq} and Proposition \ref{rathomprop} we obtain
\begin{corollary}\label{infcon}
For each odd natural number of the form $w^0w^\infty(l^0)^2$ there exists a simply connected smooth 7-manifold with infinitely many inequivalent contact structures of Sasaki type which admit cscS metrics.
\end{corollary}

These results can be compared to a very similar result in our unpublished manuscript \cite{BoTo14b} in which the contact structures are toric; whereas, those above have complexity one or two.

\def\cprime{$'$} \def\cprime{$'$} \def\cprime{$'$} \def\cprime{$'$}
  \def\cprime{$'$} \def\cprime{$'$} \def\cprime{$'$} \def\cprime{$'$}
  \def\cdprime{$''$} \def\cprime{$'$} \def\cprime{$'$} \def\cprime{$'$}
  \def\cprime{$'$}
\providecommand{\bysame}{\leavevmode\hbox to3em{\hrulefill}\thinspace}
\providecommand{\MR}{\relax\ifhmode\unskip\space\fi MR }
\providecommand{\MRhref}[2]{%
  \href{http://www.ams.org/mathscinet-getitem?mr=#1}{#2}
}
\providecommand{\href}[2]{#2}

\end{document}